\documentclass[11pt]{amsart}

\usepackage{amsmath,amssymb,enumerate}

\usepackage[T1]{fontenc}

\usepackage{babel}
\usepackage{amstext}
\usepackage{amsfonts}
\usepackage{latexsym}
\usepackage{ifthen}
\usepackage{mathrsfs}
\usepackage[all,cmtip]{xy}
\xyoption{all}
\usepackage{pstricks}
\usepackage{tikz-cd}

\usepackage{hyperref}
\hypersetup{
	colorlinks=true,
	linkcolor=red,
	citecolor=blue}

\theoremstyle{plain}

\newtheorem{theorem}{Theorem}[section]
\newtheorem{proposition}[theorem]{Proposition}
\newtheorem{lemma}[theorem]{Lemma}
\newtheorem{corollary}[theorem]{Corollary}

\newtheorem{thmA}{Theorem}

\newtheorem{corA}[thmA]{Corollary}

\theoremstyle{definition}

\newtheorem{definition}[theorem]{Definition}
\newtheorem{remark}[theorem]{Remark}

\theoremstyle{remark}

\newcommand{\reg}{\mathrm{reg}}
\newcommand{\loc}{\mathrm{loc}}

\makeatletter
  \ifnum\@ptsize=0 
  \fi
  \ifnum\@ptsize=2 
  \fi
\makeatother

\newcommand{\R}{\mathbb{R}}
\newcommand{\Q}{\mathbb{Q}}
\newcommand{\Z}{\mathbb{Z}}
\newcommand{\C}{\mathbb{C}}
\newcommand{\N}{\mathbb{N}}

\newcommand\sF{{\mathcal F}}
\newcommand\sG{{\mathcal G}}
\newcommand\sH{{\mathcal H}}
\newcommand\sI{{\mathcal I}}
\newcommand\sL{{\mathcal L}}
\newcommand\sM{{\mathcal M}}
\newcommand\sO{{\mathcal O}}

\newcommand{\dbar}{{\overline{\partial}}}
 
\DeclareMathOperator{\PSH}{\mathrm{PSH}}
\DeclareMathOperator{\codim}{\mathrm codim}
\DeclareMathOperator{\image}{\mathrm Im}
\DeclareMathOperator{\nd}{nd}
\DeclareMathOperator{\Supp}{Supp}
\DeclareMathOperator{\rk}{rk}

\begin{document}

\title{Nonvanishing results for Kähler varieties}

	\author{Andreas H\"oring}
	\address{Andreas H\"oring, Universit\'e C\^ote d’Azur, CNRS, LJAD, France}
	\email{andreas.hoering@univ-cotedazur.fr}
	
	\author{Vladimir Lazi\'c}
	\address{Vladimir Lazi\'c, Fachrichtung Mathematik, Campus, Geb\"aude E2.4, Universit\"at des Saarlandes, 66123 Saarbr\"ucken, Germany}
	\email{lazic@math.uni-sb.de}

	\author{Christian Lehn}
	\address{Christian Lehn, Fakult\"at f\"ur Mathematik, Ruhr-Universit\"at Bochum, Universit\"atsstraße 150, Postfach IB 45, 44801 Bochum, Germany}
	\email{christian.lehn@rub.de}
			

\begin{abstract}
Nonvanishing theorems play a central role in birational geometry, since they derive geometric consequences from numerical information and constitute a crucial step towards abundance and semiampleness problems. General nonvanishing statements remain rare, especially in the Kähler setting.

We present two types of nonvanishing results for compact Kähler varieties. First, on non-uniruled varieties with nonzero Euler--Poincaré cha\-rac\-te\-ris\-tic, we prove nonvanishing for adjoint bundles of numerical dimension one on Kähler klt pairs, as well as nonvanishing for nef line bundles of numerical dimension one on $K$-trivial varieties. Second, on hyperkähler manifolds we study line bundles $\mathcal L$ which are nef but not big, and establish a dichotomy: either nonvanishing holds for $\mathcal L$, or any closed positive current in the cohomology class of $\mathcal L$ has maximal Lelong components with a rather restricted geometry. We obtain much stronger abundance-type results in dimension $4$.
\end{abstract}

\maketitle

	\begingroup
		\hypersetup{linkcolor=black}
		\setcounter{tocdepth}{1}
		\tableofcontents
	\endgroup
	
\section{Introduction}
The Nonvanishing Conjecture predicts that on a minimal compact Kähler klt pair $(X,\Delta)$, the $\Q$-line bundle $K_X+\Delta$ has non-negative Kodaira dimension. This is an important part of the Abundance Conjecture, which postulates that \mbox{$K_X+\Delta$} is semiample. The Abundance Conjecture is known for Kähler $3$-fold pairs by \cite{CHP16,DO23,DO24,GP24}, and the proof builds crucially on the Nonvanishing theorem for Kähler $3$-folds from \cite{Pet01,DP03}.

In order to discuss results in higher dimensions, recall first that on a compact complex variety $X$, a nef line bundle $\sL$ has \emph{numerical dimension}
\[
\nd(\sL) := \max \big\{k \in \N \ | \ c_1(\sL)^k\neq 0 \mbox{ in } H^{2k}(X,\mathbb{R})\big\}\in\{0,1,\dots,\dim X\}.
\]
Now, if $(X,\Delta)$ is a compact Kähler klt pair such that $K_X+\Delta$ is nef, there have been only two general results on the Nonvanishing Conjecture in higher dimensions. If $\nd(K_X+\Delta)=\dim X$, then $X$ is projective by Lemma~\ref{lem:Moishezon}, and the result follows from the Basepoint free theorem \cite[Theorem~3.3]{KM98}. At the other end of the spectrum, when $\nd(K_X+\Delta)=0$, then it was shown only recently in \cite[Theorem~5.11]{Wan21} and \cite[Corollary~1.18]{CGP23} that $\kappa(X,K_X+\Delta)\geq0$. To our knowledge, there are no other general results on the Nonvanishing Conjecture on Kähler varieties which are not projective.

When the underlying variety $X$ is additionally projective, then there are two types of results known, all under the additional assumption that the Euler--Poincar\'e characteristic $\chi(X,\sO_X)$ is not zero. Unconditionally, we know that the Nonvanishing Conjecture holds in every dimension if $(X,\Delta)$ has terminal singularities and $\nd(K_X+\Delta)=1$ by \cite[Theorem~6.7]{LP18} and \cite[Theorem B]{LP20b}. Conditionally -- assuming the Minimal Model Program in lower dimensions -- we know that the Abundance Conjecture holds in every dimension if $(X,\Delta)$ is hermitian semipositive by \cite[Theorem B]{LP18} and \cite[Theorem 5.2]{LP20b}. A vast generalisation of those results, which reformulates the Abundance Conjecture as a statement on the behaviour of multiplier ideals, is \cite[Theorem~11.1]{Laz24}.

One of the main goals of this paper is to generalise several of these results to the K\"ahler setting.

\subsection*{Nonvanishing results on Kähler varieties}

The first main result of this paper is a proof of the Nonvanishing Conjecture when the numerical di\-men\-sion is $1$ and the Euler--Poincar\'e characteristic does not vanish. Parts of The\-o\-rem \ref{thm:main1} are new even when the underlying variety is projective.

\begin{thmA}\label{thm:main1}
Let $(X,\Delta)$ be a $\Q$-factorial compact Kähler klt pair such that $X$ is not uniruled. Assume that $K_X+\Delta$ is nef of numerical dimension $1$, and that $\chi(X,\sO_X)\neq0$. Then $\kappa(X,K_X+\Delta)\geq0$.

Moreover, if $\pi\colon Y\to X$ is a resolution and if $T_{\min}$ is a current with minimal singularities in the cohomology class of $\pi^*(K_X+\Delta)$, then one of the following two cases holds:
\begin{enumerate}[\normalfont (a)]
\item all Lelong numbers of $T_{\min}$ vanish, or
\item $T_{\min}$ is the only closed positive current in the cohomology class of $\pi^*(K_X+\Delta)$.
\end{enumerate}
\end{thmA}

A few remarks are in order. In case (a) of Theorem \ref{thm:main1}, if $X$ is additionally projective and if we assume the existence of good minimal models in lower dimensions, then we can say more: $K_X+\Delta$ is semiample in that case. This follows by combining Theorem \ref{thm:main1} with \cite[Theorem~1.5]{GM17}. Very recently, it was shown in \cite{LX25} that the nonvanishing result of Theorem \ref{thm:main1} implies the semiampleness of $K_X+\Delta$ when $X$ is projective if one assumes the existence of minimal models in lower dimensions. The methods of \cite{LX25} seem to work only in the projective category. However, by adapting the proof of \cite[Theorem~5.1]{GM17} in Section \ref{sec:semiampleness}, we are able to deduce the following corollary in dimension $4$.

\begin{corA}\label{cor:main1}
Let $(X,\Delta)$ be a $\Q$-factorial compact Kähler klt pair of di\-men\-si\-on $4$ such that $X$ is not uniruled. Assume that $K_X+\Delta$ is nef of numerical dimension $1$, and that $\chi(X,\sO_X)\neq0$. Let $\pi\colon Y\to X$ be a resolution and let $T_{\min}$ be a current with minimal singularities in the cohomology class of $\pi^*(K_X+\Delta)$. If all Lelong numbers of $T_{\min}$ vanish, then $K_X+\Delta$ is semiample.
\end{corA}

We also remark that in case (b) of Theorem \ref{thm:main1}, the cohomology class of $K_X+\Delta$ is rigid, in the terminology of \cite{LX24}. This demonstrates further that rigid currents should be instrumental for the proof of the Abundance Conjecture.

In fact, we prove a much more general statement in Theorem \ref{thm:nu1}, which also allows to deduce our second main result as a special case.

\begin{thmA}\label{thm:main2}
Let $X$ be a $\Q$-factorial compact Kähler klt variety such that $X$ is not uniruled and $c_1(X)=0$. Let $\sL$ be a nef line bundle on $X$ of numerical di\-men\-si\-on~$1$. If $\chi(X,\sO_X)\neq0$, then there exists a $\Q$-divisor $D\geq0$ on $X$ which belongs to the cohomology class of $\sL$.
\end{thmA}

Similarly as in Corollary~\ref{cor:main1}, we can obtain more information in dimen\-sion~$4$, see Corollary~\ref{cor:main2}.

We add a few comments concerning the proof of Theorems \ref{thm:main1} and \ref{thm:main2}. The logic of the proof is similar to that of \cite[Theorem 6.7]{LP18}, but considerably more involved. The central ingredient is a nonvanishing criterion proven in Theorem \ref{thm:nonvanishingFormsnu1}. In order to formulate it, we introduce \emph{good bilinear forms} in Section \ref{sec:good}, which distil good properties of Hodge--Riemann forms on Kähler manifolds, and of Beauville--Bogomolov--Fujiki forms on hyperkähler manifolds. We introduce in \S\ref{subsec:siudecomcurrents} the \emph{Siu decomposition} of pseudoeffective classes on compact complex manifolds, which allows us to conduct a fine analysis in the proof of Theorem \ref{thm:nu1}. The proofs of Theorems \ref{thm:main1} and \ref{thm:main2} are made possible also by the recent important result of \cite{Ou25} on the ge\-o\-met\-ry of uniruled Kähler manifolds, which generalises \cite{BDPP13} from the projective to the Kähler setting, and serves as a fundamental ingredient in the main result of \cite{CP25}, which is crucial to our argument.

\subsection*{Nonvanishing on hyperkähler manifolds}

We now turn to hyperkähler manifolds. The following theorem is our third main result.

\begin{thmA}\label{thm:main3}
Let $X$ be a hyperk\"ahler manifold of dimension $2n$, and let $\sL$ be a nef line bundle on $X$ with $q_X(\sL, \sL) = 0$ where $q_X$ is the Beauville--Bogomolov--Fujiki form on $X$. Then one of the following two cases holds:
\begin{enumerate}[\normalfont (a)]
\item $\kappa(X,\sL)\geq0$, or
\item for every closed positive current $T$ in the cohomology class of $\mathcal L$ and for every maximal Lelong component\footnote{Maximal Lelong components are introduced in Definition \ref{dfn:maxlelong} below.} $Z$ of $T$, we have
\begin{enumerate}[\normalfont (i)]
\item $n\leq\dim Z\leq 2n-2$,
\item $Z$ is coisotropic with respect to the symplectic form, and 
\item $c_1(\sL|_Z)^{\dim Z-n+1}=0$.
\end{enumerate}
\end{enumerate}
\end{thmA}

This result generalises \cite[Theorem 4.9]{LM23}, where it was proved for projective hyperk\"ahler fourfolds.\footnote{The theorem was stated in \cite{LM23} also for non-projective hyperkähler fourfolds, but the proof there quotes several results which use the projectivity assumption crucially.} Theorem \ref{thm:main3} also improves on \cite[Theorem 4.1(ii)]{Ver10}, which showed that $\kappa(X,\sL)\geq0$ if $T$ has vanishing Lelong numbers. A version of Theorem~\ref{thm:main3} was contained in the unpublished preprint \cite{Ver09}. Our arguments are very different from those in \cite{Ver09}; notably, we do not make use of the concepts of \emph{nef currents} and \emph{domination} introduced in op.\ cit. Our proof is close in spirit to that of \cite{LM23}, and we crucially use  regularisation techniques from \cite{Dem92}.

Note that a part of Theorem \ref{thm:main3} remains valid also for transcendental classes, see Theorem~\ref{thm:LelongHK}.

Finally, similarly as in Corollary~\ref{cor:main1}, we obtain more information in dimen\-si\-on~$4$. Corollary \ref{cor:main3} generalises \cite[Theo\-rem~1.6]{GM17} to the non-projective setting.

\begin{corA}\label{cor:main3}
Let $X$ be a hyperk\"ahler manifold of dimension $4$, and let $\sL$ be a nef line bundle on $X$ with $q_X(\sL, \sL) = 0$. Assume that there exists a closed positive current in the cohomology class of $\mathcal L$, all of whose Lelong numbers vanish. Then $\sL$ is semiample.
\end{corA}

\subsection*{Acknowledgements} 

We would like to thank Haidong Liu and Shin-ichi Matsumura for detailed conversations on their paper \cite{LM23}, and Omprokash Das, C\'ecile Gachet, Wenhao Ou, and Thomas Peternell for useful discussions and com\-ments on the content of the present paper.

All three authors were supported by the ANR-DFG project ``Positivity on K-trivial varieties'', ANR-23-CE40-0026 and DFG Project-ID 530132094. Höring was partially supported  by the France 2030 investment plan managed by the ANR, as part of the Initiative of Excellence of Universit\'e C\^ote d'Azur, reference ANR-15-IDEX-01. Lazi\'c was supported by the DFG Project-ID 286237555 – TRR 195. Lehn was supported by the DFG Project-ID 550535392.

\section{Preliminaries}

In this paper, all manifolds are connected. The dimension of a reducible compact complex space is defined as the maximal dimension of its irreducible components. The term \emph{variety} is understood as a complex variety, i.e.\ a reduced and irreducible Hausdorff complex space that is $\sigma$-compact. We use the convention that $d^c=\frac{1}{2\pi i}(\partial-\dbar)$, so that $dd^c=\frac{i}{\pi}\partial\dbar$.

Much of the material on complex geometry and pluripotential theory discussed here can be found in \cite{Dem12a,GZ17} or in the introductory sections of \cite{Bou04,Laz24}. We use freely the terminology from birational geometry as in \cite{KM98}.

If $X$ is a complex space, a \emph{$\Q$-line bundle} on $X$ is a formal $\Q$-linear combination of line bundles on $X$. In this article, a pair $(X,\Delta)$ consists of a complex space $X$ and a $\Q$-divisor $\Delta$ such that $K_X+\Delta$ is $\Q$-Cartier.

On a normal complex variety $X$, the canonical sheaf $\omega_X$ is a reflexive sheaf which is defined as the pushforward of the determinant bundle of the sheaf of holomorphic $1$-forms on the regular part of $X$. The variety $X$ is \emph{$\Q$-factorial} if every prime Weil divisor $D \subseteq X$ is $\Q$-Cartier and if there exists a positive integer $m$ such that the reflexive power $(\omega_X^{\otimes m})^{**}$ is a locally free sheaf. In that case, $\omega_X$ and every $\Q$-divisor on $X$ define $\Q$-line bundles. We omit the double dual when the reflexive sheaves involved are $\Q$-line bundles.

As in \cite[\S II.4]{Nak04}, we use the symbol $K_X$ in a formal way, i.e.\ we write $\omega_X\simeq\sO_X(K_X)$ even if there might not exist a canonical Weil divisor with that property. To avoid confusion with Kähler forms, we use the notation $K_X$ instead of $\omega_X$. This allows to apply additive notation: in particular, on a $\Q$-factorial pair $(X,\Delta)$, the adjoint canonical $\Q$-line bundle $\omega_X\otimes\sO_X(\Delta)$ corresponds to the formal class $K_X+\Delta$, as is standard in birational geometry. For an integer $m$ we have \mbox{$\sO_X(mK_X)\simeq\omega_X^{\otimes m}$}.

We denote the Kodaira dimension of an $\R$-Cartier divisor $D$ (respectively $\Q$-line bundle $\sL$) on a normal complex variety $X$ by $\kappa(X, D)$ (respectively $\kappa(X,\sL)$), see \cite[\S II.3]{Nak04} for details.

If $D=\sum d_i D_i$ is an $\R$-divisor on a complex space $X$, we denote by $\lfloor D\rfloor:=\sum\lfloor d_i\rfloor D_i$ the round-down of $D$.

\subsection{Bott-Chern cohomology}

If $X$ is a complex manifold and $p\in\N$, we define the Bott-Chern $(p,p)$-cohomology space $H^{p,p}_\mathrm{BC}(X,\C)$ as the quotient of the space of $d$-closed smooth $(p,p)$-forms modulo the $dd^c$-exact smooth $(p,p)$-forms, and we denote by $H^{p,p}_\mathrm{BC}(X,\R)$ the underlying real vector space. It can be shown that $H^{p,p}_\mathrm{BC}(X,\C)$ is isomorphic to the quotient of the space of $d$-closed $(p,p)$-currents modulo the $dd^c$-exact $(p,p)$-currents. If additionally $X$ is compact and K\"ahler, then $H^{p,p}_\mathrm{BC}(X,\C)$ is isomorphic to the cohomology group $H^{p,p}(X,\C)$.

If $T$ is a closed $(p,p)$-current on a complex manifold $X$, we denote by $\{T\}$ its class in $H^{p,p}_\mathrm{BC}(X,\C)$. If $T$ is a real closed $(p,p)$-current on $X$, then $\{T\}\in H^{p,p}_\mathrm{BC}(X,\R)$. If $T$ is a representative of a class $\alpha\in H^{p,p}_\mathrm{BC}(X,\R)$, we write $T\in\alpha$. If $T'\in\alpha$ is another representative, we also write $T\equiv T'$.

The situation is a little more subtle on a singular variety, see Remark~\ref{remark currents with potentials}.

\subsection{Positive currents}\label{section positivity of currents}

Let $X$ be a complex manifold of dimension $n$. A $(p,p)$-form $\varphi$ on $X$ is \emph{strongly positive} if locally it can be written as a finite non-negative linear combination of forms expressible as 
$$i\alpha_1\wedge \overline\alpha_1\wedge \ldots\wedge i\alpha_{p}\wedge \overline\alpha_{p}$$
where $\alpha_i$ are $(1,0)$-forms. A $(p,p)$-form $\varphi$ is \emph{positive} if its wedge product with any strongly positive $(n-p,n-p)$-form is a top-form with nonnegative coefficients. Strongly positive forms are always positive, and positive forms are real forms. Positivity and strong positivity of forms are pointwise pro\-per\-ti\-es which do not depend on local coordinates. We refer to \cite[\S III.1.A]{Dem12a} for further properties.

A $(p,p)$-current $T$ on a complex manifold $X$ is \emph{positive} (respectively \emph{strongly positive}) if \mbox{$T(\varphi)\geq 0$} for every smooth strongly positive (respectively positive) $(n-p,n-p)$-test form, see \cite[\S III.1.B]{Dem12a}. Strongly positive currents are always positive, and positive currents are always real. Positivity and strong positivity coincide for $(p,p)$-currents when $p\in\{0,1,n-1,n\}$.

We recall from \cite[\S~2.2.3]{Bou04} that there is a natural notion of pullback for closed positive $(1,1)$-currents.

If $D$ is a prime divisor on $X$, then we denote by $[D]$ the \emph{current of integration} over the regular part of $D$. If there is no danger of confusion, we drop the brackets and write simply $D$ for this current. We similarly define the current of integration for every effective $\R$-divisor on $X$.

\subsection{Quasi-psh functions}

In this subsection, $X$ is a complex manifold. For the precise definition of a \emph{plurisubharmonic} function or \emph{psh} function \mbox{$\varphi \colon X \to [-\infty, +\infty)$} we refer to \cite{Dem12a}. The most important property we will use here is that then we have $dd^c\varphi\geq0$ in the sense of currents. If additionally $X$ is compact, then any psh function on $X$ is constant.

A closed $(1,1)$-current $T$ on $X$ is positive if and only if for each $x\in X$ there exists an open subset $x\in U\subseteq X$ such that $T$ can locally be written as $T=dd^c \varphi$ for a psh function $\varphi$ on $U$. The function $\varphi$ is called a \emph{local potential} for $T$ on $U$.

A function $\varphi\colon X\to[-\infty,+\infty)$ on a complex manifold $X$ is \emph{quasi-psh} if it is locally equal to the sum of a psh function and a smooth function. Equivalently, $\varphi$ is locally integrable and upper semicontinuous, and there exists a smooth\footnote{If $X$ is compact Kähler, the form $\theta$ can be chosen closed.} real $(1,1)$-form $\theta$ on $X$ such that $\theta+dd^c\varphi\geq 0$ in the sense of currents. Then we also say that $\varphi$ is \emph{$\theta$-psh} and we denote the set of all $\theta$-psh functions by $\PSH(X,\theta)$. If $X$ is additionally compact, if $\theta$ is a smooth closed real $(1,1)$-form on $X$, and if $T\in\{\theta\}$ is a closed positive $(1,1)$-current, then there exists $\varphi\in\PSH(X,\theta)$, which is unique up to an additive constant, such that
$$T=\theta+dd^c\varphi.$$
Such a function $\varphi$ is called a \emph{global potential} for $T$.

If $T$ is a closed positive $(1,1)$-current on $X$ with a global potential $\varphi$, the \emph{multiplier ideal sheaf} $\mathcal I(T)\subseteq \sO_X$ is defined by 
$$ \mathcal I(T)(U):=\{f\in \sO_X(U)\mid |f|e^{-\varphi}\in L^2_{\loc}(U)\} $$ 
for every open set $U\subseteq X$. The sheaf $\mathcal I(T)$ does not depend on the choice of~$\varphi$, and it is a coherent ideal sheaf on $X$.

\subsection{Positivity properties of classes}

Let $X$ be compact complex ma\-ni\-fold, let $\omega$ be a fixed smooth positive definite $(1,1)$-form on $X$, and consider a cohomology class $\alpha\in H^{1,1}_\mathrm{BC}(X,\R)$. Then $\alpha$ is
\begin{enumerate}[\normalfont (i)]
\item \emph{pseudoeffective} if there exists a closed positive $(1,1)$-current $T\in\alpha$;
\item \emph{nef} if for each $\varepsilon>0$ there exists a smooth form $\theta_\varepsilon\in\alpha$ such that $\theta_\varepsilon\geq{-}\varepsilon\omega$;
\item \emph{big} if there exist $\varepsilon>0$ and a closed $(1,1)$-current $T\in\alpha$ such that $T\geq\varepsilon\omega$.
\end{enumerate}

The definitions do not depend on the choice of $\omega$, and they correspond to the usual notions from algebraic geometry when $X$ is projective and $\alpha$ is an algebraic class. If $\alpha\in H^{1,1}_\mathrm{BC}(X,\R)$ is a nef class on a compact complex manifold $X$, the \emph{numerical dimension of $\alpha$} is
$$ \nd(\alpha):=\max \big\{k \in \N \ | \ \alpha^k\neq 0 \mbox{ in } H^{k,k}_{\rm BC}(X,\mathbb{R})\big\}\in\{0,1,\dots,\dim X\}. $$

\begin{remark}\label{remark currents with potentials}
The notions of pseudoeffectivity, bigness, nefness, and nu\-me\-ri\-cal dimension can also be formulated for real classes on singular compact complex varieties, but the notions are somewhat more involved: one has to consider \emph{currents with local potentials}. We refer to \cite[Section 1]{Dem85} and \cite[Section 3]{HP16} for details. In particular, $H^{1,1}_\mathrm{BC}(X,\R)$ denotes classes of $(1,1)$-currents with local potentials in this context, see \cite[Remark 3.2]{HP16}. As in \cite[\S~2.2.3]{Bou04}, we may pull back closed positive $(1,1)$-currents with local potentials.
\end{remark}

By \cite[p.~631]{EGZ09} every line bundle $\sL$ on a complex variety $X$ defines the Chern class $\{\sL\}\in H^{1,1}_{\mathrm{BC}}(X,\R)$, and the construction commutes with pullbacks under holomorphic maps.

We will use the concept of rigidity of pseudoeffective classes: we say that a pseudoeffective class $\alpha\in H^{1,1}_{\rm BC}(X,\R)$ on a compact complex manifold $X$ is \emph{rigid} if it contains precisely one closed positive $(1,1)$-current. We refer to \cite{LX24} for further properties and for relevance of rigid classes to birational geometry and other fields.

\subsection{Lelong numbers}

Let $\Omega\subseteq\C^n$ be an open subset and let $\varphi$ be a psh function on $\Omega$. The \emph{Lelong number} of $\varphi$ at a point $x\in\Omega$ is
$$ \nu(\varphi,x):=\lim_{r\to0}\frac{\sup_{B(x,r)}\varphi}{\log r}. $$
The Lelong number $\nu(\varphi,x)$ does not depend on the choice of local coordinates around $x$. If $T$ is a closed positive $(1,1)$-current on a complex manifold $X$ with a local potential $\varphi$ around a point $x\in X$, we define the Lelong number of $T$ at $x$ as
$$\nu(T,x):=\nu(\varphi,x).$$
This does not depend on the choice of $\varphi$.

For $c\geq0$, we define the \emph{Lelong upperlevel sets} as 
$$E_c(T):=\{x\in X\mid \nu(T,x)\geq c\}.$$
Then a fundamental theorem of \cite{Siu74} states that for each $c>0$, the set $E_c(T)$ is a proper analytic subset of~$X$. Thus, for any analytic subset $Y$ of~$X$ we may define the \emph{generic Lelong number of $T$ along $Y$} as 
$$ \nu(T,Y):=\inf\{\nu(T,x)\mid x\in Y\};$$
this is equal to $\nu(T,x)$ for a very general point $x\in Y$.

\begin{remark}
By Skoda's lemma \cite[Lemma~5.6]{Dem01}, there is a relationship between Lelong numbers and multiplier ideals. Indeed, fix $x\in X$. If $\nu(T,x)<1$, then $\mathcal I(T)_x=\sO_{X,x}$. If $\nu(T,x)\geq\dim X+s$ for some integer $s\geq0$, then $\mathcal I(T)_x\subseteq\mathfrak{m}_x^{s+1}$ where $\mathfrak{m}_x$ is the maximal ideal in $\sO_{X,x}$.
\end{remark}

\subsection{Siu decomposition}
If $X$ is a complex manifold and if $T$ is a closed positive $(1,1)$-current on $X$, then there exist at most countably many codimension $1$ irreducible analytic subsets $D_k$ such that $T$ has the \emph{Siu decomposition}
$$T=R+\sum_k\nu(T,D_k)\cdot D_k,$$
where $R$ is a closed positive $(1,1)$-current such that $\codim_X E_c(R)\geq2$ for each $c>0$. Note that the decomposition is unique. In this paper, we call $\sum\nu(T,D_k)\cdot D_k$ the \emph{divisorial part} and $R$ the \emph{residual part} (of the Siu decomposition) of $T$.

\subsection{Minimal singularities}\label{subsec:minimalsingularities}

Let $\alpha$ be a closed real smooth $(1,1)$-form on a compact complex manifold $X$ whose class $\{\alpha\}\in H^{1,1}_\mathrm{BC}(X,\R)$ is pseudoeffective. If $T_1$ and $T_2$ are two closed positive $(1,1)$-currents in $\alpha$ with corresponding global potentials $\varphi_1$ and $\varphi_2$, we say that $T_1$ is \emph{less singular} than $T_2$, and write $T_1\preceq T_2$ if there exists a constant $C$ such that $\varphi_2\leq\varphi_1+C$. This does not depend on the choice of global potentials.

A minimal element $\varphi_{\min}\in\PSH(X,\alpha)$ with respect to the relation~$\preceq$ exists, and it is called a \emph{global potential with minimal singularities in $\PSH(X,\alpha)$}. The corresponding current
$$T_{\min}=\alpha+dd^c\varphi_{\min}$$
is a \emph{current with minimal singularities in $\{\alpha\}$}. One immediately checks that for each point $x\in X$ we have
\begin{equation}\label{eq:Tmin}
\nu(T_{\min},x)=\min_{T\in\alpha}\nu(T,x).
\end{equation}
It is also clear that for each positive number $\lambda$, the current $\lambda T_{\min}$ has minimal singularities in the class $\{\lambda\alpha\}$. All currents with minimal singularities in a fixed cohomology class have the same multiplier ideal.

\subsection{Modified nef classes}\label{subsec:modifiednef}

The following notion is due to Boucksom, see \cite[Definition~2.2]{Bou04}.

\begin{definition}
Let $X$ be a compact complex manifold with a hermitian form $\omega$ on $X$. A class $\alpha\in H^{1,1}_\mathrm{BC}(X,\R)$ is \emph{modified Kähler} if there exists $\varepsilon >0$ and a closed $(1,1)$-current $T\in \alpha$ such that $T\geq \varepsilon \omega$ and $\nu(T,D)=0$ for every prime divisor $D$ on $X$. A class $\alpha\in H^{1,1}_\mathrm{BC}(X,\R)$ is \emph{modified nef} if for every $\varepsilon >0$ there exists a closed $(1,1)$-current $T_\varepsilon\in \alpha$ such that $T_\varepsilon\geq {-}\varepsilon \omega$ and $\nu(T_\varepsilon,D)=0$ for every prime divisor $D$ on $X$.
\end{definition}

Modified nef classes appear in the \emph{Boucksom--Zariski decomposition} of pseudoeffective classes on a compact complex manifold, see \cite{Bou04} for additional properties of the decomposition. It generalises the Zariski decom\-position on surfaces. See also \S\ref{subsec:siudecomcurrents} below.

In Proposition~\ref{prop:twoexamples}, we will use that modified nef classes form a cone which is the closure of the open cone of modified Kähler classes. We will also use that in each modified Kähler class $\alpha$ there exists a closed positive current $T$ which has analytic singularities along an analytic subvariety of codimension at least $2$ in the sense of \cite[\S2.5]{Bou04}, see \cite[Remark~1]{Bou04}.

\subsection{Uniruled Kähler manifolds}

The following important theorem was proved recently in \cite[Theorem 1]{Ou25}.

\begin{theorem}\label{thm:Ou}
Let $X$ be a compact Kähler manifold. Then $X$ is uniruled if and only if $K_X$ is not pseudoeffective.
\end{theorem}

This theorem generalises \cite[Corollary 0.3]{BDPP13} to the setting of Kähler manifolds. Using the methods from \cite{Ou25}, Theorem \ref{thm:quotientspsef} below was obtained in \cite[Theorem 1.4]{CP25}, which generalises \cite[Theorem~1.2]{CP19} and \cite[Theorem 3.11]{Ver10}. We will use it crucially in the proof of Propo\-sition~\ref{pro:quotientKaehler}.

\begin{theorem}\label{thm:quotientspsef}
Let $X$ be a compact Kähler manifold. Let $(\Omega^1_X)^{\otimes m}\to Q$ be a torsion-free coherent quotient for some $m\geq1$. If $K_X$ is pseudoeffective, then $\det Q$ is pseudoeffective.
\end{theorem}

\subsection{A projectivity criterion}

We mentioned the following known result in the introduction.

\begin{lemma}\label{lem:Moishezon}
Let $X$ be a compact Kähler variety of dimension $n$ with rational singularities and let $\sL$ be a nef line bundle on $X$ with $\nd(\sL)=n$. Then $X$ is projective.
\end{lemma}

\begin{proof}
Pick a resolution $\pi\colon Y\to X$. Then $Y$ is a Kähler manifold and $\nd(\pi^*\sL)=n$, hence $Y$ is projective by the proof of \cite[Theorem~5.1]{Siu84}. Therefore, $X$ is a Moishezon variety, hence projective by \cite[Theo\-rem~1.6]{Nam02}.
\end{proof}

\subsection{Iitaka fibrations}

We will need the following lemma in Section \ref{sec:semiampleness}. The proof is the same as \cite[Lemma 2.3]{LP18} and \cite[Lemma 2.5]{Laz24a}; it is essentially contained in the proof of \cite[Theorem 2.1.33]{Laz04}, whose proof is almost the same in the analytic setting, see \cite[Theorem 5.10]{Uen75}.

\begin{lemma}\label{lem:iitaka}
Let $X$ be a compact complex variety and let $\sL$ be a $\Q$-line bundle on $X$ with $\kappa(X,\sL)\geq0$. Then for any sufficiently high resolution $\pi\colon Y\to X$ there exists a fibration $f\colon Y\to Z$ onto a projective manifold $Z$ such that \mbox{$\dim Z=\kappa(X,\sL)$}, and for every $\pi$-exceptional $\Q$-divisor $E\geq0$ on $Y$ and for a very general fibre $F$ of $f$, we have
$$\kappa\big(F,\pi^*\sL\otimes\sO_F(E)\big)=0.$$
\end{lemma}

\subsection{Dlt models}

We will need the following higher dimensional version of \cite[Theorem 6.1]{DHP24}.

\begin{theorem}\label{thm:dltmodel}
Assume the Minimal Model Program for compact Kähler dlt pairs in dimension at most $n-1$.

Let $(X, \Delta)$ be a compact $\Q$-factorial K\"ahler log canonical pair of dimen\-sion~$n$.  Then there exists a $\Q$-factorial dlt pair $(X',\Delta')$ and a projective bi\-me\-ro\-mor\-phic morphism $\pi\colon X'\to X$ such that $K_{X'}+\Delta'\sim_\Q \pi^*(K_X+\Delta)$.

In particular, if $n\leq4$, then the result holds unconditionally.
\end{theorem}

\begin{proof}
The proof is the same as that of \cite[Theorem 6.1]{DHP24}, where it was stated for $4$-folds. We only remark that the Minimal Model Program in lower dimensions is used in order to apply the special termination; the proof in the algebraic case from \cite{Fuj07a} applies verbatim. Note that the theorem holds unconditionally on $4$-folds by \cite[Theorems 1.1 and 1.2]{DH25}.
\end{proof}

\subsection{A result from mixed Hodge theory}

The following result is useful when dealing with singular varieties. The proof uses the mixed Hodge structure on the cohomology. 

\begin{lemma}\label{lemma mixed hodge lemma}
Let $X$ be a compact Kähler variety, let $Z \subseteq X$ be a closed subvariety, and let $\pi\colon\widetilde Z \to Z$ be a resolution of singularities. If a class $\alpha \in H^k(X,\C)$ satisfies $\pi^*(\alpha|_Z) = 0$, then $\alpha|_Z=0$.
\end{lemma}

\begin{proof}
Denote by $W_{k-1}H^k(Z,\C)$ the $(k-1)$-st step of the weight filtration of the mixed Hodge structure on $H^k(Z,\C)$. By \cite[Corollary~5.42]{PS08},\footnote{This corollary is stated for complex algebraic varieties, but the formalism of mixed Hodge theory is the same in the algebraic and analytic categories; one only needs the Kähler assumption (or, more generally, Fujiki's class $\mathscr C$ would be sufficient) to establish the existence of the mixed Hodge structure in the analytic case. We refer to \cite{AG05} (especially Chapter II.3) for mixed Hodge theory on compact complex spaces.} we have
$$ W_{k-1}H^k(Z,\C) = \ker\big(\pi^*\colon H^k(Z,\C)\to H^k(\widetilde Z,\C)\big). $$
Then strictness of morphisms of mixed Hodge structures, see \cite[Corollary~3.6]{PS08}, and purity of the Hodge structure on $H^k(X,\C)$ imply that the class $\alpha|_Z$ cannot be contained in $W_{k-1}H^k(Z,\C)$ unless it is zero.
\end{proof}

\section{Positive currents}

In this section, we collect several auxiliary results on positive and strongly positive currents, which we will need in the remainder of the paper.

\subsection{Siu decomposition of classes}\label{subsec:siudecomcurrents}

We define Siu decompositions of pseudoeffective $(1,1)$-classes, which will be crucial in the proof of Theorem~\ref{thm:main1}. This was implicit in \cite{Laz24}, and it is similar to the Boucksom--Zariski decomposition, but there are situations where the two decompositions differ and have applications in different contexts.

\begin{definition}\label{dfn:Siuclasses}
Let $X$ be a compact complex manifold. Let $\alpha\in H^{1,1}_{\mathrm{BC}}(X,\R)$ be a pseudoeffective class and let $T_{\min}\in\alpha$ be a current with minimal singularities. Let
$$T_{\min}=R+D$$
be the Siu decomposition of $T_{\min}$, where $D$ is the divisorial part and $R$ is the residual part. Set $D(\alpha):=\{D\}$ and $R(\alpha):=\{R\}$. Then the decomposition 
$$\alpha=R(\alpha)+D(\alpha)$$
is called the \emph{Siu decomposition of $\alpha$} into the \emph{divisorial part $D(\alpha)$} and the \emph{residual part $R(\alpha)$}. If $D(\alpha)=0$, then we say that $\alpha$ is a \emph{Siu-residual class}.
\end{definition}

Note that the Siu decomposition of a pseudoeffective class $\alpha\in H^{1,1}_{\mathrm{BC}}(X,\R)$ does not depend on the choice of $T_{\min}$ by \eqref{eq:Tmin}. Moreover, by \cite[Lemma~4.11]{LM23} or \cite[Lemma 5.8]{Laz24}, the current $D$ in Definition \ref{dfn:Siuclasses} is an $\R$-divisor.

\begin{remark}\label{rem:rigid1}
Let $X$ be a compact complex manifold and let $\alpha\in H^{1,1}_{\mathrm{BC}}(X,\R)$ be a Siu-residual pseudoeffective class. Then $\alpha$ is modified nef by \cite[Proposition 3.2(ii) and Proposition 3.6(i)]{Bou04}. Furthermore, if $\alpha$ and $\beta$ are proportional pseudoeffective classes in $H^{1,1}_{\mathrm{BC}}(X,\R)\setminus\{0\}$, then it is easy to see by \S\ref{subsec:minimalsingularities} that $\alpha$ is Siu-residual if and only if $\beta$ is Siu-residual.
\end{remark}

\begin{remark}\label{rem:rigid2}
Let $X$ be a compact complex manifold and let $\alpha\in H^{1,1}_{\mathrm{BC}}(X,\R)$ be a pseudoeffective class. Then the class $D(\alpha)$ is rigid. Indeed, let $T_{\min}\in\alpha$ be a current with minimal singularities, and let $T_{\min}=R+D$ be its Siu decomposition, where $D$ is the divisorial part. Then $D$ is a current with minimal singularities in $D(\alpha)$ by \cite[Remark 5.4]{Laz24}, hence for each closed positive current $T\in D(\alpha)$ we have $T\geq D$ by \eqref{eq:Tmin} and by considering the Siu decomposition of $T$. Therefore, $T-D\geq0$ and $T-D\equiv 0$, hence $T-D=0$ by \cite[Example 3.2]{LX24}.
\end{remark}

\begin{remark}
Let $X$ be a compact complex manifold and let $\alpha\in H^{1,1}_{\mathrm{BC}}(X,\R)$ be a pseudoeffective class. As mentioned in \S\ref{subsec:modifiednef}, we have the Boucksom--Zariski decomposition $\alpha=Z(\alpha)+N(\alpha)$ as in \cite[Definition 3.7]{Bou04}. Let us compare this to the Siu decomposition $\alpha=R(\alpha)+D(\alpha)$ from Definition~\ref{dfn:Siuclasses}. We have that $N(\alpha)$ is rigid by \cite[Theorem 3.12(i) and Proposition 3.13]{Bou04}, and $D(\alpha)$ is rigid by the previous remark. Let $N\in N(\alpha)$ and $D\in D(\alpha)$ be the unique closed positive currents; they are both $\R$-divisors. By \cite[Proposition 3.6]{Bou04}, we have $N\leq D$, with equality if $\alpha$ is a big class. Thus, the Siu decomposition and the Boucksom--Zariski decomposition coincide for big classes. They are, however, different in general: for instance, as in \cite[Exam\-ple 3.6]{LX24}, there exists a projective surface $X$ and a nef prime divisor $C$ on $X$ such that the cohomology class $\beta:=\{C\}$ is rigid. Then we have $N(\beta)=0$ and $D(\beta)=\beta$.
\end{remark}

\subsection{A rationality lemma}
We will need the following simple lemma in the proof of Theorem \ref{thm:nu1a}.

\begin{lemma}\label{lem:rationalclass}
Let $X$ be a compact complex manifold and let $\alpha\in H^{1,1}_{\mathrm{BC}}(X,\Q)$. Assume that there exists an $\R$-divisor $D \geq 0$ on $X$ such that $\alpha=\{D\}$. Then there exists a $\Q$-divisor $D'\geq0$ such that $\alpha=\{D'\}$ and $\kappa(X,D')=\kappa(X,D)$.
\end{lemma}

\begin{proof}
Let $D_1,\dots,D_r$ be the prime components of $D$ and consider the vector subspace $V:=\bigoplus_{i=1}^r \R D_i$ of the space of all divisors on $X$. Consider the cone $\mathcal C:=\sum_{i=1}^r\R_+D_i\subseteq V$ and the linear map
$$\rho\colon V\to H^{1,1}_{\mathrm{BC}}(X,\R),\quad E\mapsto \{E\}.$$
Then $\rho^{-1}\big(\rho(\alpha)\big)$ is a rational affine subspace of the real vector space $V$, hence the set $\mathcal P:=\rho^{-1}\big(\rho(\alpha)\big)\cap\mathcal C$ is cut out from $\rho^{-1}\big(\rho(\alpha)\big)$ by finitely half-spaces whose boundaries are rational hyperplanes. The set $\mathcal P$ is not empty since $D\in \mathcal P$. Thus, there exists a rational divisor $0\leq D'\in \mathcal P$ such that $\Supp D'=\Supp D$, and we clearly have $\alpha=\{D'\}$ and $\kappa(X,D')=\kappa(X,D)$.
\end{proof}

\subsection{A pseudoeffectivity lemma}
We will need the following result in the proof of Theorem \ref{thm:nonvanishingFormsnu1}.

\begin{lemma}\label{lem:psef}
Let $\pi\colon X\to Y$ be a proper bimeromorphic morphism, where $X$ and $Y$ are complex varieties with rational singularities. Let $\sF$ be a pseudoeffective line bundle on $X$ such that the reflexive sheaf $\sG:=(\pi_*\sF)^{**}$ is a line bundle on $Y$. Then $\sG$ is pseudoeffective.
\end{lemma}

\begin{proof}
The line bundles $\sF$ and $\pi^*\sG$ differ along the exceptional locus of $\pi$, hence there exist $\pi$-exceptional divisors $E_1\geq0$ and $E_2\geq0$ such that
\begin{equation}\label{eq:55}
\sF\otimes\sO_X(E_1)=\pi^*\sG\otimes\sO_X(E_2).
\end{equation}
We have $\{\sG\} \in H^{1,1}_{\mathrm{BC}}(Y,\R)$ and $\{\pi^*\sG\}=\pi^*\{\sG\} \in H^{1,1}_{\mathrm{BC}}(X,\R)$. Let $T\in\{\sF\}$ be a closed positive $(1,1)$-current with local potentials, and let $\alpha\in\{\sG\}$ be a closed $(1,1)$-form with local potentials. Then \eqref{eq:55} gives
\begin{equation}\label{eq:excep}
T+[E_1]\equiv \pi^*\alpha+[E_2].
\end{equation}
Since $\pi_*[E_1]=\pi_*[E_2]=0$ by \cite[Proposition 4.2.78]{BM19}, applying $\pi_*$ to the equation \eqref{eq:excep} we obtain
$$\pi_*T\equiv \alpha.$$
Let $Z\subseteq Y$ be the image of the $\pi$-exceptional set in $Y$. Then $Z$ is of codimension at least $2$ in $Y$. The current $(\pi_*T)|_{Y\setminus Z}$ is closed and positive, has local potentials, and satisfies $(\pi_*T)|_{Y\setminus Z}\in\{\alpha\}|_{Y\setminus Z}$, hence by the proof of \cite[Proposition~4.6.3(i)]{BG13}, there exists a closed positive current with local potentials $S\in\{\alpha\}=\{\sG\}$. Thus, $\sG$ is pseudoeffective.
\end{proof}

\subsection{A rigidity result}

Several times in the paper we will need the following lemma, which ge\-ne\-ra\-li\-ses \cite[Example 3.2]{LX24} to classes of higher bi-degree on a compact Kähler manifold. We include the proof by Simone Diverio from \cite{Div12}.

\begin{lemma}\label{lem:cohotozero}
Let $X$ be a K\"ahler manifold of dimension $n$, let $\omega$ be a Kähler form on $X$, and let $T$ be a closed positive $(p,p)$-current on $X$.
\begin{enumerate}[\normalfont (a)]
\item If $\int_X T\wedge \omega^{n-p}=0$, then $T=0$. 
\item Assume that $X$ is compact. If $\{T\}\cdot\{\omega\}^{n-p}=0$, then $T=0$. In particular, if $\{T\}=0\in H^{p,p}(X,\R)$, then $T=0$.
\end{enumerate}
\end{lemma}

\begin{proof}
We first show (a). Let $h$ be the hermitian metric on $X$ whose fundamental form is $\omega$. Let $\|T\|$ be the mass measure with respect to local orthonormal frames in the metric $h$, see \cite[Remark~III.1.15]{Dem12a}, and let $\sigma_T$ be the trace measure $\sigma_T$ with respect to $\omega$, see \cite[Definition~III.1.21]{Dem12a}. Then there exists a positive constant $C$ such that \mbox{$\|T\|\leq C\sigma_T$} by \cite[Proposition~III.1.14 and~equation~III.(1.22)]{Dem12a}. Therefore, setting $C':=\frac{C}{2^{n-p}(n-p)!}$, we have
\begin{equation}\label{eq:measures}
\|T\|(X)\leq C\sigma_T(X)=C'\int_X T\wedge \omega^{n-p}.
\end{equation}
Since the coefficients of $T$ are measures whose total variations are dominated by $\|T\|$, we conclude by \eqref{eq:measures} and the assumption of (a) that~$T=0$.

For (b), note that $\int_X T\wedge \omega^{n-p} = \{T\}\cdot \{\omega\}^{n-p}$, since the integral depends only on the cohomology class of $T$, as $\omega$ is closed. We conclude as in (a), by \eqref{eq:measures} and the assumption of (b).
\end{proof}

We will often need the following useful result.

\begin{lemma}\label{lem:cohoclasses}
Let $X$ be a compact K\"ahler manifold of dimension $n$, let $\eta$ be a K\"ahler class on $X$, let $T$ be a closed positive $(p,p)$-current on $X$, and let $0\leq m \leq n-p$ be an integer. Let $\alpha_1,\dots,\alpha_m$ be nef classes in $H^{1,1}(X,\R)$. Then the following statements hold.
\begin{enumerate}[\normalfont (a)]
\item There exists a closed positive $(p+m,p+m)$-current $\Theta\in \{T\}\cdot \alpha_1\cdot\ldots\cdot\alpha_m$.
\item We have $\{T\}\cdot \alpha_1\cdot\ldots\cdot\alpha_m \cdot \eta^{n-m-p} \geq 0$, with equality if and only if $\{T\}\cdot \alpha_1\cdot\ldots\cdot\alpha_m=0$.
\end{enumerate}
\end{lemma}

\begin{proof}
We first show (a), and we argue similarly as in the proof of \cite[Proposition~6.21]{Dem12}. For every $\varepsilon>0$ and each $1\leq i\leq m$, the class $\alpha_i+\varepsilon\eta$ is Kähler. We fix Kähler forms $\omega_{i,\varepsilon}\in\alpha+\varepsilon\eta$ and a Kähler form $\omega\in\eta$. Then each current 
$$\Theta_\varepsilon:=T\wedge \omega_{1,\varepsilon}\wedge\ldots\wedge\omega_{m,\varepsilon}\in \{T\}\cdot(\alpha_1+\varepsilon\eta)\cdot\ldots\cdot(\alpha_m+\varepsilon\eta)$$
is positive by \cite[Corollary~III.1.16]{Dem12a}. Since
$$\int_X \Theta_\varepsilon\wedge \omega^{n-m-p}=\{T\}\cdot(\alpha_1+\varepsilon\eta)\cdot\ldots\cdot(\alpha_m+\varepsilon\eta)\cdot\eta^{n-m-p},$$
and since the last expression is uniformly bounded, we conclude by \cite[Proposition~III.1.23]{Dem12a} that there exists a sequence of positive real numbers $\{\varepsilon_n\}_{n\in\N}$ converging to $0$ such that there exists a weak limit $\Theta$ of the sequence of currents $\{\Theta_{\varepsilon_n}\}_{n\in\N}$. Clearly, $\Theta$ is a closed positive current in the class $\{T\}\cdot \alpha_1\cdot\ldots\cdot\alpha_m$.

For (b), note that by (a) we have
$$ \{T\}\cdot \alpha_1\cdot\ldots\cdot\alpha_m \cdot\eta^{n-m-p}=\{\Theta\} \cdot \eta^{n-m-p}=\int_X \Theta\wedge\omega^{n-m-p}\geq0 .$$
If equality holds, then $\Theta=0$ by Lemma \ref{lem:cohotozero}(b). This implies the result.
\end{proof}

\subsection{Positivity of cohomology classes}

We include the proof of the following well-known result for completeness and for the lack of explicit reference.

\begin{lemma}\label{lem:pseflines}
Let $X$ be a compact complex manifold. Then the cone of pseudo\-effective classes $\mathcal E(X)\subseteq H^{1,1}_{\mathrm{BC}}(X,\R)$ contains no affine lines.
\end{lemma}

\begin{proof}
Assume that $\mathcal E(X)$ contains an affine line. Then there exist classes $\alpha,\beta\in H^{1,1}_{\mathrm{BC}}(X,\R)$ such that $\beta\neq0$ and $\alpha+t\beta\in\mathcal E(X)$ for all $t\in\R$. Since the cone $\mathcal E(X)$ is closed, we conclude that
$$\beta=\lim_{t\to\infty}\frac{1}{t}(\alpha+t\beta)\in\mathcal E(X)\quad\text{and}\quad {-}\beta=\lim_{t\to\infty}\frac{1}{t}(\alpha-t\beta)\in\mathcal E(X).$$
In particular, there exist closed positive $(1,1)$-currents $T$ in the class $\beta$ and $S$ in the class ${-}\beta$. Therefore, $T+S\equiv0$, hence 
\begin{equation}\label{eq:300}
T+S=0
\end{equation}
by \cite[Example 3.2]{LX24}. We will derive a contradiction by showing that $T=S=0$, and hence $\beta=0$.

To this end, note first that the problem is local by \cite[Theo\-rem~7]{dRh84}. Hence, we may assume that $X$ is a coordinate ball. From~\eqref{eq:300} we immediately obtain that
\begin{equation}\label{eq:300a}
T(\varphi)=S(\varphi)=0
\end{equation}
for every strongly positive test $(n-1,n-1)$-form $\varphi$ on $X$, where $n:=\dim X$. For the remainder of the proof, we fix a test $(n-1,n-1)$-form $\theta$. We need to show that $T(\theta)=S(\theta)=0$. 

Assume first that there exist a smooth function $f$ with compact support and a strongly positive form $\varphi$ on $X$ such that $ \theta=f\varphi$. Treating the real and imaginary parts of $f$ separately, we may assume that $f$ is a real function. Fix a smooth real function $\chi$ on $X$ with compact support such that $\chi\equiv1$ on the compact subset $K\subseteq X$ which is the support of $f$. Fix a constant $C\geq0$ such that $f(z)\geq{-C}$ for $z\in K$, and consider the functions $h:=C\chi$ and $g:=h+f=\chi(C+f)$. Then $g$ and~$h$ are smooth \emph{nonnegative} functions with compact support on $X$ such that $f=g-h$. Therefore, the test forms $ \theta_1:= g\varphi$ and $ \theta_2:= h\varphi$ are strongly positive, and we have $\theta=\theta_1-\theta_2$. Thus, $T(\theta)=S(\theta)=0$ by \eqref{eq:300a}.

In the general case, by \cite[Lem\-ma~III.1.4]{Dem12a} there exists a set $\{\varphi_1,\dots,\varphi_m\}$ of strongly positive forms which generate the $C^\infty(X,\C)$-module $C^\infty\big(X,\bigwedge^{n-1,n-1}T_X^*\big)$. Then there are smooth functions $f_1,\dots,f_m$ with compact support such that $ \theta=f_1\varphi_1+\dots+f_m\varphi_m $. Thus, we conclude by the case treated above.
\end{proof}

We will need the following two results in Section \ref{sec:semiampleness}.

\begin{lemma}\label{lem:numdim}
Let $X$ be a compact Kähler variety. If $\alpha$ and $\beta$ are nef classes in $H^{1,1}(X,\R)$ such that the class $\beta-\alpha$ is pseudoeffective, then
$$\nd(\alpha)\leq\nd(\beta).$$
\end{lemma}

\begin{proof}
By passing to a resolution, we may assume that $X$ is a compact Kähler manifold. Let $k:=\nd(\alpha)$ and let $\eta$ be a Kähler class on $X$. Then by Lemma \ref{lem:cohoclasses}(b) we have
$$ \alpha^k\cdot\eta^{n-k}>0 $$
and
$$(\beta^k-\alpha^k)\cdot\eta^{n-k}=\sum_{i=0}^{k-1}(\beta-\alpha)\cdot\alpha^i\cdot\beta^{k-1-i}\cdot\eta^{n-k}\geq0.$$
These inequality imply that
$$\beta^k\cdot\eta^{n-k}=\alpha^k\cdot\eta^{n-k}+(\beta^k-\alpha^k)\cdot\eta^{n-k}>0,$$
which shows that $\beta^k\neq0$, as desired.
\end{proof}

\begin{lemma}\label{lem:nefexceptional}
Let $f\colon X\dashrightarrow Y$ be a bimeromorphic map from a compact Kähler manifold $X$ to a normal compact complex variety $Y$. Assume that the map $f^{-1}$ does not contract divisors. Let $\alpha\in  H^{1,1}(X,\R)$ be a nef class and assume that there exist an $f$-exceptional $\R$-divisor $E_1\geq0$ and an $\R$-divisor $E_2\geq0$ on $X$ with $E_1-E_2\in\alpha$. Then $E_1=E_2$. 
\end{lemma}

\begin{proof}
By taking a resolution of indeterminacies of the map $f$, then applying \cite[Corollary 2]{Hir75}, and then the resolution of singularities \cite[Theo\-rem~1.10]{BM97}, we conclude that there exists a bimeromorphic morphism \mbox{$g\colon Z\to X$} from a compact Kähler manifold $Z$, such that the induced map $\pi:=f\circ g$ is a projective morphism from $Z$ to $Y$, and the $\R$-divisors $F_1:=g^*E_1$ and $F_2:=g^*E_2$ have simple normal crossings. Then $F_1$ is $\pi$-exceptional and we have $F_1-F_2\in\pi^*\alpha$. It suffices to show that $F_1=F_2$. By cancelling common components, we may assume that $F_1$ and $F_2$ have no common components. It remains to show that $F_1=F_2=0$. We set $D:=F_1-F_2$.

Assume that $F_1\neq0$. Then, by \cite[Lemma III.5.1]{Nak04}, there exists a component $\Gamma$ of $F_1$ such that the class $\{F_1\}|_\Gamma$ is not pseudoeffective on $\Gamma$. On the other hand, the class 
$$\{F_1\}|_\Gamma=\{F_2\}|_\Gamma+\{D\}|_\Gamma=\{F_2|_\Gamma\}+\{D\}|_\Gamma$$
is pseudoeffective, since $F_2|_\Gamma$ is an effective $\R$-divisor and $\{D\}|_\Gamma$ is a nef class. This is a contradiction which shows that $F_1=0$.

Therefore, $D$ is a nef $\R$-divisor such that ${-}D=F_2$ is effective. By Lem\-ma~\ref{lem:pseflines}, this is only possible if $\{F_2\}=0$, hence $F_2=0$ by Lemma \ref{lem:cohotozero}(b). This finishes the proof.
\end{proof}

\subsection{On strong positivity}

The following lemma is well known, but we could not find an explicit reference.

\begin{lemma}\label{lem:localpositivity}
Let $X$ be a complex manifold and let $T$ be a current on $X$. Let $\{U_i\}_{i\in I}$ be an open covering of $X$.  Then $T$ is positive (respectively strongly positive) if and only if for each $i\in I$ the current $T|_{U_i}$ is positive (respectively strongly positive).
\end{lemma}

\begin{proof}
The proof is analogous to that of \cite[Theorem 7]{dRh84}. One direction is clear. For the other direction, assume that each $T|_{U_i}$ is strongly positive; the case of positive currents is completely analogous. Let $\eta$ be a positive test form on $X$ and let $K$ be a compact subset of $X$ containing the support of $\eta$. Let $\{\psi_i\}_{i\in I}$ be a partition of unity subordinate to the cover $\{U_i\}_{i\in I}$ as in \cite[Corollary 2]{dRh84}. Then $K$ meets the supports of only finite many $\psi_i$; we denote these by $\psi_1,\dots,\psi_m$ and the corresponding members of the open covering by $U_1,\dots,U_m$. Since each $\psi_i\eta$ is a positive test form on $U_i$, we have $T(\psi_i\eta)\geq0$ by assumption. Therefore,
$$T(\eta)=T\bigg(\sum_{i=1}^m\psi_i\eta\bigg)=\sum_{i=1}^m T(\psi_i\eta)\geq0,$$
as desired.
\end{proof}

We are grateful to Shin-ichi Matsumura for communicating the proof of the following lemma which will be used in the proof of Theorem \ref{thm:LelongHK}.

\begin{lemma}\label{lem:strongly1}
Let $X$ be a complex manifold and let $T$ be a strongly positive $(p,p)$-current on $X$. Let $Z$ be an irreducible analytic subvariety of $X$ of codimension $p$ and let $0\leq\nu\leq\nu(T,Z)$. Then the $(p,p)$-current $T-\nu[Z]$ is strongly positive.
\end{lemma}

\begin{proof}
Since the current of integration $[Z]$ is strongly positive by \cite[III.\S1.20 and Theorem III.2.7]{Dem12a}, and since
$$T-\nu[Z]=\big(T-\nu(T,Z)[Z]\big)+\big(\nu(T,Z)-\nu\big)[Z],$$
we may assume that $\nu=\nu(T,Z)$.

By \cite[Proposition III.8.16]{Dem12a}, we have $\mathbf{1}_Z T=\nu [Z]$, hence $T - \nu [Z]=\mathbf{1}_{X \setminus Z} T$. It suffices to show that the current $\mathbf{1}_{X \setminus Z} T$ is strongly positive. The problem is local by Lemma \ref{lem:localpositivity}, so we may assume that $X$ is a coordinate ball. Let $\{\chi_k\}_{k\geq1}$ be smooth real functions on $X$ with compact support, whose supports exhaust $X\setminus Z$, and which converge pointwise to $\mathbf{1}_{X \setminus Z}$. Let~$\eta$ be a positive test form on $X$. Since the coefficients of $\mathbf{1}_{X \setminus Z} T$ are Radon measures, we have
$$ \mathbf{1}_{X \setminus Z} T(\eta)=\lim_{k\to\infty}\mathbf{1}_{X \setminus Z} T(\chi_k\eta)=\lim_{k\to\infty}T(\chi_k\eta) \geq 0. $$
This finishes the proof.
\end{proof}

We will need the following result several times in the paper.

\begin{lemma}\label{lem:strongly2}
Let $X$ be a compact complex manifold of dimension $n$ and let $T$ be a strongly positive $(p,p)$-current on $X$. Let $\alpha\in H^{1,1}_\mathrm{BC}(X,\R)$ be a pseudoeffective class, let $\theta\in\alpha$ be a smooth form, and let $\varphi\in\PSH(X,\theta)$ such that the unbounded locus of $\varphi$ is contained in an analytic set of dimension at most $n-p-1$. Set $T':=\theta+dd^c\varphi$. Then the wedge product 
$$T\wedge T':=T\wedge\theta+dd^c(\varphi T)$$
is a well-defined strongly positive $(p+1,p+1)$-current.
\end{lemma}

\begin{proof}
The wedge product is well defined by \cite[Corollary III.4.10]{Dem12a}. To show that it is strongly positive, we follow the proof of \cite[Proposition~III.3.2]{Dem12a}. The problem is local by Lemma \ref{lem:localpositivity}, so we may assume that $X$ is a coordinate ball and that $T'=dd^c\varphi$, where $\varphi$ is psh on $X$. Pick a smoothing kernel $\rho\in C^\infty(X,\R)$, set $\rho_k(x):=k^{2n}\rho(kx)$, and consider the convolution $\varphi_k:=\varphi*\rho_k$ for $k\in\N_{>0}$. Then $\{\varphi_k\}_{k>0}$ is a decreasing family of smooth psh functions on $X$ such that $\displaystyle\lim_{k\to\infty}\varphi_k=\varphi$ pointwise, hence
\begin{equation}\label{eq:09a}
\lim_{k\to\infty}dd^c(\varphi_kT)=dd^c(\varphi T)
\end{equation}
by \cite[Corollary III.4.3]{Dem12a}. Since each $\varphi_k$ is smooth, we have $dd^c(\varphi_kT)=dd^c\varphi_k\wedge T$, hence $dd^c(\varphi_kT)$ is strongly positive for each $k$ by \cite[Corollary~III.1.16]{Dem12a}. Therefore, \eqref{eq:09a} yields that $T\wedge T'=dd^c(\varphi T)$ is the weak limit of a sequence of strongly positive currents, thus it is strongly positive.
\end{proof}

\section{Semiampleness results}\label{sec:semiampleness}

In this section, we prove semiampleness results which will be used in the proofs of Corollaries \ref{cor:main1}, \ref{cor:main2} and \ref{cor:main3}, and are also of independent interest. 

We start with the following theorem, which generalises \cite[Theo\-rem~7.3]{Kaw85b} and \cite[Theorem 1.35]{DO24}. The proof is essentially the same; the main difference is that we invoke Lemma \ref{lem:nefexceptional}, allowing us to bypass an argument relying on restriction to a complete intersection surface, which does not work in the Kähler setting.

\begin{theorem}\label{thm:kappa>0}
Assume the existence of minimal models and the Abundance Conjecture for compact Kähler klt pairs in dimension at most $n-1$.

Let $(X, \Delta)$ be a compact $\Q$-factorial K\"ahler klt pair of dimension $n$ such that $K_X+\Delta$ is nef. If $\kappa(X, K_X+\Delta)\geq 1$, then $K_X+\Delta$ is semiample.

In particular, if $n\leq4$, then the result holds unconditionally.
\end{theorem}

\begin{proof}
The last statement of the theorem follows from the first one by \cite[Theorems~1.1 and 1.2]{DH25}, \cite[Theorem 1.1]{DO24}, and \cite[Theorem 1.2]{DO23}. In the remainder of the proof, we show the first statement.

\medskip

\emph{Step 1.}
In this step, we reduce to the case
\begin{equation}\label{eq:200f}
\kappa(X, K_X+\Delta)\leq n-2.
\end{equation}

Assume first that $\kappa(X,K_X+\Delta)\geq n-1$. If $\nd(K_X+\Delta)=n$, then $X$ is projective by Lemma~\ref{lem:Moishezon}, and $K_X+\Delta$ is nef and big, so we conclude by the Basepoint free theorem \cite[Theorem 3.3]{KM98}. If $\nd(K_X+\Delta)=n-1$, then we conclude by \cite[Theorem 5.5]{Nak87} and \cite[Theorem 4.8]{Fuj11c}.
        
\medskip

\emph{Step 2.}
By Lemma \ref{lem:iitaka}, there exists a log resolution $\pi\colon Y\to X$ of the pair $(X,\Delta)$ and a fibration $f\colon Y\to Z$ to a projective manifold $Z$ such that \mbox{$\dim Z=\kappa(X,K_X+\Delta)$}, and such that for every $\pi$-exceptional $\Q$-divisor $E\geq0$ on $Y$ and for a very general fibre $F$ of $f$, we have
\begin{equation}\label{eq:200a}
\kappa\big(F,(\pi^*(K_X+\Delta)+E)|_F\big)=0.
\end{equation}
Fix such a fibre $F$ of $f$ for the remainder of the proof. Then $\dim F\geq2$ by~\eqref{eq:200f} and $F$ is smooth by generic smoothness.

\medskip

\emph{Step 3.}
We show in this step that
\begin{equation}\label{eq:600}
\pi^*(K_X+\Delta)|_F\sim_\Q 0.
\end{equation}

To this end, we may write
\begin{equation}\label{eq:200}
K_Y+\Delta_Y\sim_\Q\pi^*(K_X+\Delta)+E_1
\end{equation}
where $\Delta_Y$ and $E_1$ are effective $\Q$-divisors without common components. In particular, $E_1$ is $\pi$-exceptional. If we denote $\Delta_F:=\Delta_Y|_F$, then
$$K_F + \Delta_F\sim_\Q \big(\pi^*(K_X+\Delta)+E_1\big)|_F$$
by \eqref{eq:200} and by adjunction, and the pair $(F,\Delta_F)$ is klt. In particular, we have
\begin{equation}\label{eq:200b}
\kappa(F,K_F+\Delta_F)=0
\end{equation}
by \eqref{eq:200a}.

Since $\dim Z=\kappa(X,K_X+\Delta)\geq1$, we have $\dim F\leq n-1$. By assumption, there exists a minimal model $\varphi\colon (F,\Delta_F)\dashrightarrow (F',\Delta_{F'})$ such that $K_{F'}+\Delta_{F'}$ is semiample. Since $\kappa(F',K_{F'}+\Delta_{F'})=\kappa(F,K_F+\Delta_F)=0$ by \eqref{eq:200b}, we con\-clude that
\begin{equation}\label{eq:200c}
K_{F'}+\Delta_{F'}\sim_\Q 0.
\end{equation}

Consider a resolution of indeterminacies $(p,q)\colon \widehat{F}\to F\times F'$ of the map~$\varphi$. Then by the Negativity lemma \cite[Lemma 1.3]{Wan21} and by \eqref{eq:200c}, there exists a $q$-exceptional $\Q$-divisor $\widehat E\geq0$ such that
$$ p^*(K_F+\Delta_F)\sim_\Q q^*(K_{F'}+\Delta_{F'})+\widehat E\sim_\Q \widehat E. $$
If we denote $E_2:=p_*\widehat{E}$, then $E_2$ is $\varphi$-exceptional and
$$ K_F+\Delta_F\sim_\Q E_2. $$
This and \eqref{eq:200} give
$$\pi^*(K_X+\Delta)|_F\sim_\Q E_2-E_1|_F,$$
which together with Lemma \ref{lem:nefexceptional} implies \eqref{eq:600}.

\medskip

\emph{Step 4.}
In this final step, we finish the proof. Since $\kappa(X,K_X+\Delta)\geq1$, there exists a $\Q$-divisor $D\geq0$ on $Y$ such that $D\sim_\Q \pi^*(K_X+\Delta)$. Hence, by Step~3 and by \cite[Lemma 1.34]{DO24}, there exists an ample $\Q$-divisor $H$ on $Z$ such that $D\leq f^*H$. Therefore, by \cite[Pro\-po\-si\-tion 2.10]{Nak87} and by Lemma~\ref{lem:numdim}, we have
\begin{align*}
\nd(K_X+\Delta)&=\nd\big(\pi^*(K_X+\Delta)\big)=\nd(D)\leq \nd(f^*H)=\kappa(Y, f^*H)\\
&=\kappa(Z, H)=\dim Z=\kappa(X, K_X+\Delta)\leq\nd(K_X+\Delta).
\end{align*}
Thus, $\kappa(X, K_X+\Delta)=\nd(K_X+\Delta)$, and we conclude by \cite[Theo\-rem~5.5]{Nak87} and \cite[Theorem 4.8]{Fuj11c}.
\end{proof}

The following theorem generalises \cite[Theorem 5.1]{GM17} to the Kähler setting, and the proof follows that of op.\ cit.\ very closely. We include a de\-ta\-iled proof for the benefit of the reader.

\begin{theorem}\label{thm:GM}
Assume the Minimal Model Program and the Abundance Conjecture for compact Kähler klt pairs in dimension at most $n-1$.

Let $(X, \Delta)$ be a compact $\Q$-factorial K\"ahler klt pair of dimension $n$ such that $K_X+\Delta$ is nef and $\kappa(X,K_X+\Delta)\geq0$. Let $\pi\colon Y \to X$ be a resolution and assume that there exists a closed positive current $T\in\{\pi^*(K_X+\Delta)\}$, all of whose Lelong numbers vanish. Then $K_X+\Delta$ is semiample.

In particular, if $n\leq4$, then the result holds unconditionally.
\end{theorem}

\begin{proof}
The last statement of the theorem follows from the first one by \cite[Theorems~1.1 and 1.2]{DH25}, \cite[Theorem 1.1]{DO24}, and \cite[Theorem 1.2]{DO23}. In the remainder of the proof we show the first statement.

If $\kappa(X,K_X+\Delta)\geq1$, then we conclude by Theorem \ref{thm:kappa>0}. Thus, we may assume that
\begin{equation}\label{eq:100}
\kappa(X,K_X+\Delta)=0.
\end{equation}

\medskip

\emph{Step 1.}
By hypothesis, there exists a $\Q$-divisor $D\geq0$ such that
$$K_X+\Delta\sim_\Q D.$$
If $D=0$, we are done. Therefore, we may assume that $D\neq0$, and set
$$\ell:=\max\big\{t\geq0\mid (X,\Delta+tD)\text{ is log canonical}\big\}\in\Q.$$
By Theorem \ref{thm:dltmodel}, there exists a $\Q$\nobreakdash-factorial Kähler dlt pair $(Z,\Delta_Z)$ and a pro\-jec\-tive bimeromorphic morphism $\varphi\colon Z\to X$ such that 
$$K_Z+\Delta_Z\sim_\Q\varphi^*(K_X+\Delta+\ell D).$$
Set $S:=\lfloor\Delta_Z\rfloor$ and $B:=\Delta_Z-S$; note that $S\neq0$ since the pair $(X, \Delta+lD)$ is log canonical but not klt. Then we have
\begin{equation}\label{eq:101}
K_Z+S+B\sim_\Q (1+\ell)\varphi^*D.
\end{equation}

\medskip

\emph{Step 2.}
We will prove in Step 3 that the restriction map
\begin{equation}\label{eq extension of sections}
\rho_m\colon H^0\big(Z, m(K_Z+S+B)\big) \to H^0\big(S, m(K_S+B|_S)\big)
\end{equation}
is surjective for each sufficiently divisible positive integer $m$. In this step, we show how this surjectivity implies the theorem.

Indeed, note that 
$$\kappa(Z,K_Z+S+B)=\kappa(X,D)=0$$
by \eqref{eq:100} and \eqref{eq:101}. Note further that $S \subseteq \Supp\varphi^*D$ since $S$ maps onto a log canonical centre of the pair $(X,\Delta+\ell D)$, which must lie in $D$ as $(X,\Delta)$ is klt. These two facts immediately give that for each sufficiently divisible $m$, the map $\rho_m$ is the zero-morphism. Therefore, the claim implies that
\begin{equation}\label{eq:102}
h^0\big(S, m(K_S+B|_S)\big)=0\quad\text{for all }m\text{ sufficiently divisible.}
\end{equation}
On the other hand, the pair $(S, B|_S)$ is a Kähler slc pair such that $K_S+B|_S$ is nef, since $K_Z+S+B$ is nef by \eqref{eq:101}. But then $K_S+B|_S$ is semiample by the assumptions of the theorem and by \cite[Theorem 1.1]{Fuj25}, which contradicts~\eqref{eq:102}. 

\medskip

\emph{Step 3.}
It remains to prove the claim about surjectivity of the restriction map $\rho_m$ from \eqref{eq extension of sections}. By blowing up further, we may assume that the morphism $\pi$ factors through $\varphi$. We denote by  $f\colon Y\to Z$ the induced morphism. Note that the condition of the theorem on the Lelong numbers of the current $T$ is preserved by taking pullbacks by \cite[Corollary 4]{Fav99}.

There exist effective $\Q$-divisors $S_Y$, $B_Y$ and $E_Y$, pairwise without common components, such that
\begin{equation}\label{eq:103}
K_Y+S_Y+B_Y\sim_\Q f^*(K_Z+S+B)+E_Y,
\end{equation}
where $S_Y=\lfloor S_Y+B_Y\rfloor$, and the divisor $E_Y$ is $f$-exceptional. Then, as in Step 2, we have $S_Y \subseteq \Supp\pi^*D$, and \eqref{eq:103} gives
\begin{equation}\label{eq:104}
K_Y+S_Y+B_Y\sim_\Q \pi^*D+E_Y
\end{equation}
and
\begin{equation}\label{eq:105}
K_{S_Y}+B_Y|_{S_Y}\sim_\Q (f|_{S_Y})^*(K_S+B|_S)+E_Y|_{S_Y}.
\end{equation}

Now, fix a positive integer $m$ such that $m(K_Z+S+B)$, $mD$ and $mE_Y$ are are Cartier. Let $s\in H^0(Y,mE_Y)$ be a section whose zero-divisor is $mE_Y$. By \eqref{eq:103} and \eqref{eq:105}, we have the commutative diagram
		\begin{center}
			\begin{tikzcd}
				H^0\big(Z,m(K_Z+S+B)\big) \arrow[rr, "\psi"] \arrow[d, "\rho_m" swap] && H^0\big(Y,m(K_Y+S_Y+B_Y)\big) \arrow[d, "\rho_m'"] \\
				H^0\big(S, m(K_S+B|_S)\big) \arrow[rr, "\psi'"] & & H^0\big(S_Y,m(K_{S_Y}+B_Y|_{S_Y})\big) ,
			\end{tikzcd}
		\end{center}
where the vertical maps are restrictions, the map $\psi$ sends a section $t$ to~$s\cdot f^*t$, and the map $\psi'$ sends a section $t'$ to $s|_{S_Y}\cdot (f|_{S_Y})^*t'$. Note that $\psi$ is an isomorphism, and $\psi'$ is injective by the projection formula, since we have $(f|_{S_Y})_*\sO_{S_Y}=\sO_{S}$ by \cite[Theorem 5.48]{KM98}.

Pick $u\in H^0\big(S, m(K_S+B|_S)\big)$. Then, by~\eqref{eq:104} and by \cite[Corol\-la\-ry~4.2]{GM17}, the section
$$\psi'(u)\in H^0\big(S_Y,m(K_{S_Y}+B_Y|_{S_Y})\big)$$
extends to a section
$$U'\in H^0\big(Y,m(K_Y+S_Y+B_Y)\big);$$
the assumption on the existence of a current $T$ with vanishing Lelong numbers as in the statement of the theorem is used here. Therefore, the preimage $U\in H^0\big(Z,m(K_Z+S+B)\big)$ of $U'$ under the iso\-mor\-phism $\psi$ is a section such that $\rho_m(U)=u$, as desired.
\end{proof}

\section{Hyperkähler preliminaries}

A \emph{hyperkähler manifold} is a simply connected compact Kähler manifold whose space of global holomorphic $2$-forms is spanned by a holomorphic symplectic form. It is always of even dimension.

\subsection{Beauville--Bogomolov--Fujiki form}

We recall that by \cite{Bea83}, on a hyperkähler manifold $X$ of dimension $2n$ there exists a unique non-degenerate primitive integral bilinear form 
$$q_X\colon H^2(X,\R)\times H^2(X,\R)\to \R$$
of signature $(1,h^{1,1}(X)-1)$ on $H^{1,1}(X,\R)$ which satisfies the \emph{Fujiki relation}: there exists a positive rational number $c_X$ depending only on the deformation class of $X$ such that for each $\alpha \in H^2(X,\R)$ we have
$$q_X(\alpha,\alpha)^n = c_X \int_X\alpha^{2n}.$$
We refer to $q_X$ as the \emph{Beauville--Bogomolov--Fujiki form}, or \emph{BBF form} for short.

We include the proof of the following known lemma for the sake of completeness.

\begin{lemma}\label{lem:Fujikirelation}
Let $X$ be a hyperkähler manifold of dimension $2n$ and let $q_X$ be the BBF form on $X$. Then there exists a positive rational number $c$ such that for all $\alpha_1,\dots,\alpha_{2n}\in H^2(X,\R)$ we have
$$ \alpha_1\cdot\ldots\cdot\alpha_{2n}=c\sum_{\rho\in S_{2n}}q_X(\alpha_{\rho(1)},\alpha_{\rho(2)})\cdot q_X(\alpha_{\rho(3)},\alpha_{\rho(4)})\cdot\ldots\cdot q_X(\alpha_{\rho(2n-1)},\alpha_{\rho(2n)}). $$
\end{lemma}

\begin{proof}
For $\alpha_1,\ldots,\alpha_{2n} \in H^2(X,\R)$ and $t_1,\dots,t_{2n}\in\R$, set 
$$\alpha_t:=t_1 \alpha_1 + \ldots + t_{2n}\alpha_{2n}.$$
Then by the Fujiki relation there exists $c_X\in\Q_{>0}$ such that
$$q_X(\alpha_t,\alpha_t)^n = c_X \int_X \alpha_t^{2n}.$$
The lemma follows by comparing coefficients with $t_1\cdot\ldots\cdot t_{2n}$ on both sides of this equation.
\end{proof}

The following corollary is also a consequence of \cite[Corollary~2.15]{Ver09}. We provide the details for the reader's convenience.

\begin{corollary}\label{cor:wedge}
Let $X$ be a hyperkähler manifold of dimension $2n$ and let~$q_X$ be the BBF form on $X$. Let $\sigma\in H^2(X,\R)$ be the class of a symplectic $2$-form on $X$ and let $\beta\in H^{1,1}(X,\R)$ be a class with $q_X(\beta,\beta)=0$. Then for any $2 \leq p\leq n$ and any classes $\omega_1,\dots,\omega_{p-2}\in H^{1,1}(X,\R)$ on $X$ we have
$$\beta^p \cdot \sigma^{n-p+1} \cdot \overline{\sigma}^{n-p+1}\cdot \omega_1\cdot\ldots\cdot\omega_{p-2}=0\in H^{2n,2n}(X,\R).$$
\end{corollary}

\begin{proof}
Set
$$\alpha_i:=\begin{cases}\beta& \text{for }i=1,\dots,p,\\
\sigma& \text{for }i=p+1,\dots,n+1,\\
\overline{\sigma}&\text{for }i=n+2,\dots,2n-p+2,\\
\omega_{i+p-2n-2}& \text{otherwise}.
\end{cases}$$
We have to show that $\alpha_1\cdot\ldots\cdot\alpha_{2n}=0$. To that end, fix a permutation $\rho\in S_{2n}$. We claim that there exist a pair $(2k-1,2k)$ with $1\leq k\leq n$, such that
\begin{equation}\label{eq:767}
q_X(\alpha_{\rho(2k-1)},\alpha_{\rho(2k)})=0.
\end{equation}
The claim implies the result immediately by Lemma \ref{lem:Fujikirelation}. To prove the claim, consider any $k$ such that $\beta\in\{\alpha_{\rho(2k-1)},\alpha_{\rho(2k)}\}$. If \mbox{$\{\alpha_{\rho(2k-1)},\alpha_{\rho(2k)}\}=\{\beta\}$}, then the equation \eqref{eq:767} holds by assumption. If $\{\alpha_{\rho(2k-1)},\alpha_{\rho(2k)}\}=\{\beta,\sigma\}$ or if $\{\alpha_{\rho(2k-1)},\alpha_{\rho(2k)}\}=\{\beta,\overline{\sigma}\}$, then \eqref{eq:767} holds for degree reasons. Otherwise, for each such $k$ there exists $1\leq i\leq p-2$ with $\{\alpha_{\rho(2k-1)},\alpha_{\rho(2k)}\}=\{\beta,\omega_i\}$, but this is impossible since $\beta$ occurs $p$ times in the cup product in the statement of the result.
\end{proof}

\subsection{Symplectic linear algebra}

We collect several well-known concepts from symplectic linear algebra. As a general reference for this subsection, we refer to \cite[Chapter~I, Part~1]{LM87}.

\begin{definition}
Let $V$ be a finite-dimensional complex vector space and let $\sigma \in \bigwedge^2 V^*$ be an alternating bilinear form on $V$. Consider the map
$$ \lambda\colon V \to V^*, \quad v \mapsto \sigma(v, \cdot). $$
The \emph{rank} of $\sigma$ is 
$$ \rk(\sigma) := \dim V - \dim \ker(\lambda). $$
The form $\sigma$ is \emph{symplectic} if $\rk(\sigma)=\dim V$. If~$\sigma$ is symplectic, we call $(V,\sigma)$ a \emph{symplectic vector space} and denote it also by~$V$.
\end{definition}

The rank of a $2$-form is of even dimension by the following proposition, see \cite[Theorem I.2.3]{LM87}.

\begin{proposition}\label{prop:rank}
Let $V$ be a finite-dimensional complex vector space and let $\sigma \in \bigwedge^2 V^*$. Then $\rk(\sigma)$ is even. More precisely, there exist linearly independent vectors $v^1, v^2, \dots, v^{2k}\in V^*$ such that
$$ \sigma = v^1 \wedge v^2 + v^3 \wedge v^4 + \dots + v^{2k-1} \wedge v^{2k}, $$
in which case $\operatorname{rk}(\sigma) = 2k$.
\end{proposition}

Let $(V,\sigma)$ be a symplectic vector space of dimension $2n$ and let $W$ be a subspace of $V$. Then the set
$$ W^\perp := \{ v \in V \mid \sigma(v, w) = 0 \text{ for all } w \in W \} $$
is the \emph{$\sigma$-orthogonal complement} of $W$. We have
$$\dim W+\dim W^\perp=\dim V.$$

\begin{definition}
Let $(V,\sigma)$ be a symplectic vector space of dimension~$2n$. A subspace $W \subseteq V$ of a symplectic vector space is
\begin{enumerate}[\normalfont (a)]
\item \emph{coisotropic} if $ W^\perp \subseteq W$,
\item \emph{Lagrangian} if $ W^\perp = W$.
\end{enumerate}
\end{definition}

\begin{remark}\label{rem:coisotropic}
With notation from the previous definition, assume that $W$ is coisotropic. Then $\dim W\geq n$, and it is Lagrangian if and only if $\dim W = n$.
\end{remark}

\begin{theorem} \label{theorem:coisotropic-characterisation}
Let $(V,\sigma)$ be a symplectic vector space of dimension $2n$ and let $W$ be a  subspace of $V$ of dimension $p$. Let $\sigma_W:=\sigma|_{W\times W}$ be the form induced by $\sigma$ on $W$. Then the following conditions are equivalent:
\begin{enumerate}[\normalfont (a)]
    \item $W$ is coisotropic,
    \item $\rk(\sigma_W) = 2(p-n)$,
	\item $\sigma_W^{p-n+1} = 0$.
\end{enumerate} 
\end{theorem}

\begin{proof}
The equivalence of (a) and (b) follows from \cite[Proposition I.6.5]{LM87}. Part (b) implies (c) by Proposition \ref{prop:rank}. Now assume (c). In particular, we have $p\geq n$, and hence $\rk(\sigma_W) \geq 2(p-n)$ by \cite[Corollary I.6.3]{LM87}. On the other hand, Proposition \ref{prop:rank} implies that $\rk(\sigma_W) \leq 2(p-n)$. This shows~(b).
\end{proof}

Now, let $X$ be a hyperkähler manifold with a holomorphic symplectic form~$\sigma$. A subvariety $Z \subseteq X$ is called \emph{coisotropic} if at all smooth points $z \in Z$ the subspace $T_{Z,z} \subseteq T_{X,z}$ is coisotropic with respect to $\sigma$.

\section{Good bilinear forms}\label{sec:good}

In order to uniformise approaches to nonvanishing problems in different contexts, we introduce bilinear forms with special properties.

\begin{definition}\label{dfn:goodbilinear}
Let $X$ be a compact complex manifold. A symmetric bilinear form 
$$Q\colon H^{1,1}_{\mathrm{BC}}(X,\R)\times H^{1,1}_{\mathrm{BC}}(X,\R)\to\R$$
is \emph{good} if the following two conditions are satisfied.
\begin{enumerate}[\normalfont (a)]
\item The form $Q$ has signature $\big(1,h^{1,1}(X)-1\big)$.
\item If $\alpha,\beta\in H^{1,1}_{\mathrm{BC}}(X,\R)$ are classes such that $\alpha$ is modified nef and $\beta$ is pseudoeffective, then $Q(\alpha,\beta)\geq0$.
\end{enumerate}
\end{definition}

The next proposition provides two main examples of good bilinear forms which are relevant for this paper.

\begin{proposition}\label{prop:twoexamples}
Let $X$ be a compact Kähler manifold of dimension $n$.
\begin{enumerate}[\normalfont (a)]
\item Let $\eta$ be a Kähler class on $X$ and consider the symmetric bilinear form 
$$Q_\eta\colon H^{1,1}(X,\R)\times H^{1,1}(X,\R)\to\R,\quad (\alpha,\beta)\mapsto \alpha\cdot\beta\cdot\eta^{n-2}.$$
Then $Q_\eta$ is a good bilinear form.
\item Assume that $X$ is a hyperkähler manifold and let $q_X$ be the BBF form on $X$. Then the restriction of $q_X$ to $H^{1,1}(X,\R)$ is a good bilinear form.
\end{enumerate}
\end{proposition}

\begin{proof}
It is a well-known consequence of the Hodge--Riemann bilinear relations that $Q_\eta$ has signature $\left(1,h^{1,1}(X)-1\right)$ on $H^{1,1}(X,\R)$, and we recall the proof here for convenience. The proof of the analogous statement for $q_X$ is in \cite[Corollary 23.11]{GHJ03}. Denote by $H^{1,1}(X,\R)_{\mathrm{prim}}$ the primitive part of $H^{1,1}(X,\R)$ and let $L_\eta$ denote the Lefschetz operator associated to $\eta$. Then, by the Hard Lefschetz theorem, we have the decomposition
$$H^{1,1}(X,\R)=L_\eta H^0(X,\R)\oplus H^{1,1}(X,\R)_{\mathrm{prim}}.$$
The form $Q_\eta$ is clearly positive definite on $L_\eta H^0(X,\R)\simeq \R\eta$, and it is negative definite on $H^{1,1}(X,\R)_{\mathrm{prim}}$ by \cite[Theorem 6.32]{Voi02}.

Now, let $\alpha,\beta\in H^{1,1}(X,\R)$ be classes such that $\alpha$ is modified nef and $\beta$ is pseudoeffective. We need to show that $Q_\eta(\alpha,\beta)\geq0$. The proof is similar to that of \cite[Proposition 4.2]{Bou04}, where the analogous statement was shown for $q_X$. Since the modified nef cone is the closure of the modified Kähler cone, by continuity it suffices to prove the claim when $\alpha$ is modified Kähler. By \S\ref{subsec:modifiednef} there exists a closed positive current $T\in\alpha$ with analytic singularities along an analytic subvariety of codimension at least $2$. If $S\in\beta$ is any closed positive current, then $T\wedge S\in\alpha\cdot\beta$ exists and it is a closed strongly positive $(2,2)$-current by Lemma~\ref{lem:strongly2}. Let $\omega\in\eta$ be a Kähler form. Since the form $\omega^{n-2}$ is strongly positive by \cite[Proposition III.1.11]{Dem12a}, we have
$$Q_\eta(\alpha,\beta)=\int_X T\wedge S\wedge \omega^{n-2}\geq0,$$
as desired.
\end{proof}

The first properties of good bilinear forms are summarised in the following result.

\begin{proposition}\label{HITgood}
Let $X$ be a compact complex manifold with a good bilinear form $Q\colon H^{1,1}_{\mathrm{BC}}(X,\R)\times H^{1,1}_{\mathrm{BC}}(X,\R)\to\R$. Then the following statements hold.
\begin{enumerate}[\normalfont (a)]
\item If $Q(\alpha,\alpha)>0$ and $Q(\alpha,\beta)=0$ for some $\alpha,\beta\in H^{1,1}_{\mathrm{BC}}(X,\R)$, then $Q(\beta,\beta)\leq0$, with equality if and only if $\beta=0$.
\item Let $\alpha,\beta\in H^{1,1}_{\mathrm{BC}}(X,\R)$ be classes with
$$Q(\alpha,\alpha)=Q(\beta,\beta)=Q(\alpha,\beta)=0.$$ 
Then $\alpha$ and $\beta$ are proportional.
\end{enumerate}
\end{proposition}

\begin{proof}
Part (a) follows immediately from the definition of good bilinear forms.

Now we show (b). We may assume that $\alpha\neq0$ and $\beta\neq0$, since otherwise the statement is trivial. Fix a class $\omega\in H^{1,1}_{\mathrm{BC}}(X,\R)$ with $Q(\omega,\omega)>0$, and set $a:=Q(\alpha,\omega)$ and $b:=Q(\beta,\omega)$. Then $a\neq0$ and $b\neq0$ by (a), and set $\gamma:=a\beta-b\alpha$. Then
$$Q(\gamma,\gamma)=a^2Q(\beta,\beta)-2abQ(\alpha,\beta)+b^2Q(\alpha,\alpha)=0$$
and
$$Q(\gamma,\omega)=aQ(\beta,\omega)-bQ(\alpha,\omega)=0,$$
hence $\gamma=0$ again by (a).
\end{proof}

The following is the main result of this section.

\begin{theorem} \label{thm:nu1a}
Let $X$ be a compact complex manifold with a good bilinear form $Q\colon H^{1,1}_{\mathrm{BC}}(X,\R)\times H^{1,1}_{\mathrm{BC}}(X,\R)\to\R$, and let $\alpha\in H^{1,1}_{\mathrm{BC}}(X,\R)$ be a modified nef class with $Q(\alpha,\alpha)=0$. Assume that there exist a pseudoeffective class $\beta\in H^{1,1}_{\mathrm{BC}}(X,\R)$ and a non-zero $\R$-divisor $D \geq 0$ on $X$ such that
\begin{equation}\label{eq:thesum}
\{D\}+\beta = \alpha.
\end{equation}
Then the following statements hold.
\begin{enumerate}[\normalfont (a)]
\item There exists an $\R$-divisor $E\geq0$ on $X$ such that
$$\alpha=\{E\}\quad\text{and}\quad \kappa(X,E)\geq\kappa(X,D).$$
If additionally $\alpha\in H^{1,1}_{\mathrm{BC}}(X,\Q)$, then we may choose $E$ to be a $\Q$-divisor.
\item If $R(\beta)\neq0$,\footnote{Recall the Siu decomposition of the class $\beta$ as in Definition \ref{dfn:Siuclasses}.} then there exists $t>0$ such that $R(\beta)=t\alpha$. In particular, $\alpha$ is Siu-residual.
\end{enumerate}
\end{theorem} 

\begin{proof}
The second statement in (a) follows immediately from the first statement in (a) and from Lemma \ref{lem:rationalclass}. In the remainder of the proof, we prove (b) and the first statement in (a).

There exists an $\R$-divisor $N\geq0$ such that the divisorial part in the Siu decomposition of $\beta$ is equal to $\{N\}$. Let $\rho$ denote the residual part in the Siu decomposition of $\beta$. Then $\beta = \rho + \{N\} $, and hence
\begin{equation}\label{eq:14a}
\alpha = \rho + \{N+D\}.
\end{equation}
We may assume that $\rho \neq 0$, since otherwise the result is trivial. We have
$$0=Q(\alpha,\alpha)=Q(\alpha,\rho)+Q(\alpha,\{N+D\}),$$
hence both summands are zero by the condition (b) in Definition \ref{dfn:goodbilinear}. In particular,
\begin{equation}\label{eq:12aa}
Q(\alpha,\alpha)=Q(\alpha,\rho)=0.
\end{equation}
If $Q(\rho,\rho)>0$, then this together with \eqref{eq:12aa} and Theorem \ref{HITgood}(a) implies that $\alpha=0$, which contradicts \eqref{eq:thesum} and the fact that $D\neq0$. Therefore, $Q(\rho,\rho)\leq0$. On the other hand, we have $Q(\rho,\rho)\geq0$ by the condition (b) in Definition~\ref{dfn:goodbilinear} again, hence 
\begin{equation}\label{eq:12bb}
Q(\rho,\rho)=0.
\end{equation}
Then \eqref{eq:12aa}, \eqref{eq:12bb} and Theorem \ref{HITgood}(b) yield that there exists a real number $t > 0$ with $\rho=t\alpha$, which shows (b) by Remark \ref{rem:rigid1}.

In particular, \eqref{eq:14a} gives
$$(1-t)\alpha=\{D+N\}.$$
Note that $D \neq 0$ implies that $t < 1$. Therefore, setting 
$$E:=\frac{1}{1-t}(N+D),$$
we obtain $\alpha=\{E\}$ and $\kappa(X,E)\geq\kappa(X,D)$, as desired.
\end{proof}

\section{Nonvanishing results}

In this section, we prove Theorems \ref{thm:main1} and \ref{thm:main2}. We follow the strategy from \cite{LP18,LP20b,Laz24}, but the implementation in the setting of this paper is considerably more involved.

Theorem \ref{thm:nonvanishingFormsnu1} below is the main technical result of the section, which generalises \cite[Theorem 6.3]{LP18} and \cite[Theorem 7.1]{LP20b}.

\begin{theorem}\label{thm:nonvanishingFormsnu1}
Let $X$ be a $\Q$-factorial compact Kähler variety which is not uniruled, and let $\sM$ be a nef $\Q$-line bundle on $X$. Let $t$ be a positive integer such that $\sM^{\otimes t}$ is a line bundle, and assume that the $\Q$-line bundle $\sM\otimes\sO_X({-}K_X)$ is pseudoeffective. Let $\pi\colon Y\to X$ be a resolution of $X$ and let 
$$Q\colon H^{1,1}(Y,\R)\times H^{1,1}(Y,\R)\to\R$$
be a good bilinear form. Assume that $Q\big(\{\pi^*\sM\},\{\pi^*\sM\}\big)=0$ and that for some positive integer $p$ we have  
$$ H^0\left(Y,(\Omega^1_Y)^{\otimes p} \otimes \pi^*\sM^{\otimes tm}\right) \neq 0$$
for infinitely many integers $m$. Then there exists a $\Q$-divisor $D\geq0$ such that $\{\sM\}=\{D\}$.
\end{theorem}

In order to prove Theorem \ref{thm:nonvanishingFormsnu1}, we start with the following result which generalises \cite[Lemma 4.1]{LP18}. The logic of the proof is the same as in that of op.\ cit., but the details are significantly more intricate.

\begin{lemma} \label{lemmafund} 
Let $Y$ be a compact complex manifold and let $\mathcal E$ be a locally free sheaf on $Y$. Let $\mathcal L$ be a pseudoeffective line bundle on $Y$ which is not cohomologically trivial, and assume that there exists an infinite set $\mathcal T\subseteq\Z$ such that
$$ H^0(Y,\mathcal E \otimes \mathcal L^{\otimes m}) \neq 0\quad\text{for every }m\in\mathcal T.$$
Then there exist a positive integer $r$, a saturated line bundle $\mathcal G$ in $\bigwedge^r\mathcal E$ and an infinite set $\mathcal S\subseteq\N$ such that
$$H^0(Y,\mathcal G\otimes \mathcal L^{\otimes m}) \neq 0\quad\text{for every }m\in\mathcal S.$$
\end{lemma}

\begin{proof}
\emph{Step 1.}
Denote $Z=\mathbb P(\mathcal E)$ with the projection $\pi\colon Z\to Y$. First note that 
$$H^0(Y,\mathcal E\otimes \mathcal L^{\otimes m})\simeq H^0(Z,\sO_Z(1)\otimes \pi^*\mathcal L^{\otimes m}).$$
Since $\mathcal L$ is pseudoeffective and not cohomologically trivial, the line bundle $\sO_Z(1)\otimes \pi^*\mathcal  L^{\otimes m}$ is not pseudoeffective for $m\ll0$ by Lemma \ref{lem:pseflines}. Hence, there are only finitely many negative integers in $\mathcal{T}$. Therefore, we may assume that $\mathcal T\subseteq\N$.

\medskip

\emph{Step 2.}
Fix a nontrivial section of $H^0(Y,\mathcal E \otimes \mathcal L^{\otimes m})$ for every $m \in \mathcal T$. It gives an inclusion 
$$f_m\colon \mathcal L^{\otimes -m} \to \mathcal E \quad\text{for }m\in \mathcal T.$$
For each subset $\mathcal R\subseteq\mathcal T$, consider the induced map
$$f_{\mathcal R}\colon\bigoplus_{m\in \mathcal R} \mathcal L^{\otimes -m} \to \mathcal E,$$
and set
$$\mathcal F_{\mathcal R}:=\image(f_{\mathcal R}).$$
Then we claim that $\mathcal F_{\mathcal R}$ is a coherent subsheaf of $\mathcal E$ for each subset $\mathcal R\subseteq\mathcal T$. Indeed, for each finite set $\mathcal R'\subseteq\mathcal R$ the sheaf $\mathcal F_{\mathcal R'}$ is coherent by \cite[Consequence 2, p.~237]{GR84}, hence the claim follows from \cite[Noether Lemma, p.~111]{GR84}.

\medskip

\emph{Step 3.}
By Step 2 and by \cite[Theorem, p.~92]{GR84}, for each subset $\mathcal R\subseteq\mathcal T$ the sheaf $\mathcal F_{\mathcal R}$ is locally free away from a closed analytic subset of $Y$. Let $\mathfrak A$ be the set of all subsets $\mathcal R\subseteq\mathcal T$ of the following form: either $|\mathcal R|\leq\rk\mathcal E$, or there exists $k\in\N$ such that $\mathcal R=\mathcal T\cap [k,+\infty)$. Then the set $\mathfrak A$ is countable, and we conclude that there exists a point $y\in Y$ such that for each $\mathcal R\in\mathfrak A$,
$$\mathcal F_{\mathcal R}\text{ is locally free at }y,\text{ and }\rk\mathcal F_{\mathcal R}=\rk\mathcal F_{\mathcal R,y}.$$
We fix such a point $y\in Y$ for the remainder of the proof.

\medskip

\emph{Step 4.}
Let $\mathcal R\in\mathfrak A$ and let $r:=\rk\mathcal F_{\mathcal R}$. In this step, we show that there exists a subset $\mathcal Q\subseteq\mathcal R$ such that $r=|\mathcal Q|=\rk\mathcal F_{\mathcal Q}$.

To that end, for each $m\in\mathcal R$ consider the induced $\sO_{Y,y}$-module ho\-mo\-mor\-phism 
$$f_{m,y}\colon \mathcal L_y^{\otimes -m}\to \mathcal E_y,$$
and let $\beta_m\in\mathcal L_y^{\otimes -m}$ be a generator of the $\sO_{Y,y}$-module $\mathcal L_y^{\otimes -m}$. Set 
$$\alpha_m:=f_{m,y}(\beta_m)\in\mathcal E_y.$$
Then the set $\{\alpha_m\mid m\in\mathcal R\}$ generates the locally free $\sO_{Y,y}$-module $\mathcal F_{\mathcal R,y}$ of rank $r$. If $\mathfrak m_y\subseteq\sO_{Y,y}$ is the maximal ideal, the $\sO_{Y,y}/\mathfrak m_y$-vector space $V_{\mathcal R,y}:=\mathcal F_{\mathcal R,y}/\mathfrak m_y\mathcal F_{\mathcal R,y}$ is of dimension $r$ by \cite[Theorem 2.3]{Mat89}, and it is generated by the classes $[\alpha_m]\in V_{\mathcal R,y}$ induced by the elements $\alpha_m$ for $m\in\mathcal R$. Then there exist a subset $\mathcal Q:=\{m_1,\dots,m_r\}\subseteq\mathcal R$ such that the classes $[\alpha_{m_1}],\dots,[\alpha_{m_r}]$ form a basis of $V_{\mathcal R,y}$, hence by \cite[Theorem 2.3 and Exercise 2.4]{Mat89} the elements $\alpha_{m_1},\dots,\alpha_{m_r}$ form a basis of the locally free $\sO_{Y,y}$-module $\mathcal F_{\mathcal R,y}$. Therefore, $\rk\mathcal F_{\mathcal Q}=r$, as desired.

\medskip

\emph{Step 5.}
In this step, we prove that there exist an infinite subset $\mathcal W\in\mathfrak A$, a positive integer $r$, and infinitely many subsets $\mathcal Q_i\subseteq\mathcal W$, for $i\in\N$, such that
$$r=|\mathcal Q_i|=\rk\mathcal F_{\mathcal Q_i}=\rk\mathcal F_{\mathcal W}\quad\text{for all }i\in\N.$$
The proof is by descending induction on $r$. Set $r:=\rk\mathcal F_{\mathcal T}$. If there are infinitely many subsets $\mathcal Q_i\subseteq\mathcal T$, for $i\in\N$, such that $r=|\mathcal Q_i|=\rk\mathcal F_{\mathcal Q_i}$ for all $i\in\N$, then we set $\mathcal W:=\mathcal T$.

Otherwise, there are only finitely many subsets $\mathcal Q_1,\dots,\mathcal Q_\ell\subseteq\mathcal T$, such that $r=|\mathcal Q_i|=\rk\mathcal F_{\mathcal Q_i}$ for all $i=1,\dots,\ell$. Fix $k\in\N$ such that $k$ is larger than any element in $\bigcup_{i=1}^\ell\mathcal Q_i$, and set $\mathcal T':=\mathcal T\cap [k,+\infty)$. Then $\mathcal T'\in\mathfrak A$, and we claim that $\rk\mathcal F_{\mathcal T'}<r$. Indeed, if $\rk\mathcal F_{\mathcal T'}=r$, then by Step 4 there exists a subset $\mathcal Q\subseteq\mathcal T'$ such that $r=|\mathcal Q|=\rk\mathcal F_{\mathcal Q}$. But clearly $\mathcal Q\neq\mathcal Q_i$ for all $i=1,\dots,\ell$, which is a contradiction. We replace $\mathcal T$ by $\mathcal T'$ and continue the procedure analogously, which must end after finitely many steps.

\medskip

\emph{Step 6.}
Fix an infinite subset $\mathcal W\in\mathfrak A$, a positive integer $r$, and infinitely many subsets $\mathcal Q_i\subseteq\mathcal W$, for $i\in\N$, as in Step 5. Then each map
\begin{equation}\label{eq:inclusion}
f_{\mathcal Q_i}\colon\bigoplus_{m\in \mathcal Q_i} \mathcal L^{\otimes -m} \to \mathcal F_{\mathcal Q_i}
\end{equation}
is a surjective morphism between coherent sheaves of rank $r$, hence it is injective on a complement of a closed analytic subset of $Y$ by \cite[Theorem, p.~92]{GR84} and \cite[Theorem 2.4]{Mat89}. Thus, $f_{\mathcal Q_i}$ is injective on the whole $Y$ since the sheaf $\bigoplus_{m\in \mathcal Q_i} \mathcal L^{\otimes -m}$ is torsion-free. Taking determinants in \eqref{eq:inclusion} yields inclusions
\begin{equation}\label{eq:inclusion2}
\mathcal L^{\otimes {-}q_i} \to \det\mathcal F_{\mathcal Q_i}\subseteq\det\mathcal F_{\mathcal W}\subseteq \bigwedge^r\mathcal E\quad\text{for each }i\in\N,
\end{equation}
where $q_i$ is the sum of all elements in $\mathcal Q_i$. Let $\mathcal G$ be the the saturation of $\det\mathcal F_{\mathcal W}$ in $\bigwedge^r\mathcal E$. Then $\mathcal G$ is a line bundle on $Y$ by \cite[Lemma 1.1.15]{OSS11}. The set $\mathcal S:=\{q_i\mid i\in\N\}\subseteq\N$ is infinite, and by \eqref{eq:inclusion2} we have
$$H^0(Y,\mathcal G\otimes \mathcal L^{\otimes m}) \ne 0\quad\text{for all }m\in\mathcal S,$$
which proves the result.
\end{proof}

\begin{proposition}\label{pro:quotientKaehler}
Let $Y$ be a compact Kähler manifold that is not uniruled. Let $\mathcal E$ be a positive tensor power of $\Omega_Y^1$, let $\mathcal G$ be a saturated line bundle in~$\mathcal E$, and let $\sL$ be a line bundle on $X$. Assume that there exists an infinite set $\mathcal S\subseteq\N$ and Weil divisors $N_m\geq0$ for $m\in\mathcal S$ such that
\begin{equation}\label{eq:infmanym}
\mathcal O_Y(N_m)\simeq \mathcal G\otimes \sL^{\otimes m} \quad\text{for every }m\in\mathcal S.
\end{equation}
Then there exist a positive integer $\ell$ and a pseudoeffective line bundle $\sF$ on~$Y$ such that
$$\sF \simeq \sL^{\otimes m}\otimes \sO_Y(\ell K_Y-N_m) \quad\text{for every }m\in\mathcal S.$$
\end{proposition}

\begin{proof}
The proof is almost the same as that of \cite[Proposition 4.2]{LP18}, which we recall here for completeness. If we set $\mathcal Q:=\mathcal E/\mathcal G$, then from the exact sequence $0\to\mathcal G\to\mathcal E\to\mathcal Q\to 0$ we obtain
\begin{equation}\label{eq:ell}
\det\mathcal E\simeq\mathcal G\otimes\det\mathcal Q.
\end{equation}
The sheaf $\mathcal Q$ is torsion-free since $\mathcal G$ is saturated in $\mathcal E$, hence $\mathcal F :=\det \mathcal Q$ is pseudoeffective by Theorems \ref{thm:Ou} and \ref{thm:quotientspsef}. There exists a positive integer $\ell$ such that $\det\mathcal E\simeq \sO_Y(\ell K_Y)$, and the result follows from \eqref{eq:infmanym} and \eqref{eq:ell}.
\end{proof}

Putting all this together, we have:

\begin{proof}[Proof of Theorem \ref{thm:nonvanishingFormsnu1}]
The proof follows the logic of that of \cite[The\-o\-rem~10.1]{Laz24}, with appropriate modifications.

\medskip

\emph{Step 1.}
If $\{\sM\}=0$, then the claim is obvious. Therefore, from now on we may assume that $\{\sM\}\neq0$, hence
\begin{equation}\label{eq:57b}
\{\pi^*\sM\}\neq0.
\end{equation}
We apply Lemma \ref{lemmafund} with $\mathcal E := (\Omega^1_Y)^{\otimes p} $ and $\mathcal L := \pi^*\sM^{\otimes t}$. Then there exist a positive integer $r$, a saturated line bundle $\mathcal G$ in $\bigwedge^r\mathcal E$, an infinite set $\mathcal S\subseteq\N$, and integral divisors $N_m\geq0$ for $m\in\mathcal S$ such that 
$$ \sO_Y(N_m) \simeq \mathcal G\otimes \mathcal L^{\otimes m}\quad\text{for all }m\in\mathcal S. $$
Since $Y$ is not uniruled by assumption, Proposition~\ref{pro:quotientKaehler} implies that there exist a positive integer $\ell$ and a pseudoeffective line bundle $\sF$ such that
\begin{equation}\label{eq:manyrelations}
\sF \simeq \pi^*\sM^{\otimes tm}\otimes\sO_Y(\ell K_Y-N_m) \quad\text{for all }m\in\mathcal S.
\end{equation}

\medskip

\emph{Step 2.}
Since $X$ is $\Q$-factorial and $\sM$ is a $\Q$-line bundle by assumption, there exists a positive integer $k$ such that the sheaves $\sM^{\otimes k}$, $\sO_X(k\pi_*N_m)$, and $\sO_X(kK_X)$ are line bundles, and set $\sF_X:=(\pi_*\sF^{\otimes k})^{**}$. Taking the $k$-th tensor power in \eqref{eq:manyrelations}, applying $\pi_*$ and taking the double dual, we obtain
\begin{equation}\label{eq:57}
\sF_X \simeq \sM^{\otimes ktm}\otimes \sO_X(k\ell K_X-k\pi_*N_m) \quad\text{for all }m\in\mathcal S.
\end{equation}
In particular, $\sF_X$ is a line bundle on $X$, hence it is pseudoeffective by Lemma~\ref{lem:psef}. Thus, the line bundle
$$\sH:=\sF_X\otimes\big(\sM^{\otimes k\ell}\otimes \sO_X({-}k\ell K_X)\big)$$
is pseudoeffective, since $\sM^{\otimes k\ell}\otimes \sO_X({-}k\ell K_X)$ is pseudoeffective by as\-sump\-tion. Then \eqref{eq:57} gives
$$ \sO_X(k\pi_*N_m)\otimes\sH \simeq \sM^{\otimes k(tm+\ell)} \quad\text{for all }m\in\mathcal S.$$
Pulling back these relations to $Y$, we obtain
\begin{equation}\label{eq:57a}
\sO_X(k\pi^*\pi_*N_m)\otimes\pi^*\sH \simeq \pi^*\sM^{\otimes k(tm+\ell)} \quad\text{for all }m\in\mathcal S.
\end{equation}

\medskip

\emph{Step 3.}
Assume first that there exist two distinct $m_1,m_2\in\mathcal S$ such that $\pi^*\pi_*N_{m_1}=\pi^*\pi_*N_{m_2}=0$. Then \eqref{eq:57a} gives
$$\pi^*\sH\simeq \pi^*\sM^{\otimes k(tm_1+\ell)}\simeq \pi^*\sM^{\otimes k(tm_2+\ell)}.$$
Therefore, $\pi^*\sM^{\otimes kt(m_1-m_2)}\simeq\sO_Y$, and hence $\{\pi^*\sM\}=0$, which contradicts \eqref{eq:57b}.

\medskip

\emph{Step 4.}
Therefore, we may assume that there exists $m\in\mathcal S$ such that $\pi^*\pi_*N_m\neq0$. Since $\pi^*\pi_*N_m\geq0$ and the line bundle $\pi^*\sH$ is pseudoeffective, by Theorem \ref{thm:nu1a} there exists a $\Q$-divisor $E\geq0$ on $Y$ such that $\{\pi^*\sM\}=\{E\}$. Therefore, $\{\sM\}=\{\pi_*E\}$, as desired.
\end{proof}

The first consequence of Theorem \ref{thm:nonvanishingFormsnu1} is the following result, whose proof uses methods of that of \cite[Corollary 4.5]{LP18}.

\begin{theorem}\label{thm:nu1algsing}
Let $X$ be a $\Q$-factorial compact Kähler variety which is not uniruled. Let $\sM$ be a nef $\Q$-line bundle on $X$, and assume the $\Q$-line bundle $\sM\otimes\sO_X({-}K_X)$ is pseudoeffective. Let $\pi\colon Y\to X$ be a resolution of $X$, and let $Q\colon H^{1,1}(Y,\R)\times H^{1,1}(Y,\R)\to\R$ be a good bilinear form. Assume that:
\begin{enumerate}[\normalfont (a)]
\item $Q\big(\{\pi^*\sM\},\{\pi^*\sM\}\big)=0$,
\item $\chi(Y,\sO_Y)\neq0$, and
\item there exist a closed positive current $T\in\{\pi^*\sM\}$ and a $\Q$-divisor $G\geq0$ on $Y$ such that $\mathcal I(mT)=\sO_Y({-}mG)$ for all sufficiently divisible positive integers $m$. 
\end{enumerate}
Then there exists a $\Q$-divisor $D\geq0$ such that $\{\sM\}=\{D\}$.
\end{theorem}

\begin{proof}
Assume for contradiction that such a $\Q$-divisor $D$ does not exist. Then Theorem \ref{thm:nonvanishingFormsnu1} yields that for all $p\geq 0$ and for all sufficiently divisible $m>0$ we have
$$ H^0\big(Y,\Omega^p_Y \otimes \pi^*\sM^{\otimes m}\big)=0,$$
and thus
$$H^0\big(Y,\Omega^p_Y\otimes \pi^*\sM^{\otimes m}\otimes\mathcal I(mT)\big) = 0.$$
Then \cite[Theorem 0.1]{DPS01} implies that for all $p\geq 0$ and for all sufficiently divisible $m>0$ we have
$$H^p\big(Y,\sO_Y(K_Y)\otimes \pi^*\sM^{\otimes m}\otimes\mathcal I(mT)\big) = 0.$$
Hence, by (c) we have
$$H^p\big(Y,\sO_Y(K_Y)\otimes \pi^*\sM^{\otimes m}\otimes\sO_Y({-}mG)\big) = 0$$
for all $p\geq 0$ and for all sufficiently divisible $m>0$. This implies
\begin{equation}\label{eq:89}
\chi\big(Y,\sO_Y(K_Y)\otimes \pi^*\sM^{\otimes m}\otimes\sO_Y({-}mG)\big) = 0
\end{equation}
for all sufficiently divisible $m>0$. Since the Euler--Poincar\'e characteristic $\chi\big(Y,\sO_Y(K_Y)\otimes \pi^*\sM^{\otimes m}\otimes\sO_Y({-}mG)\big)$ is a polynomial in $m$ by the Hirzebruch--Riemann--Roch theorem, \eqref{eq:89} implies that it must be identically zero, hence $\chi\big(Y,\sO_Y(K_Y)\big) = 0$ by setting $m=0$. Equivalently, by Serre duality we have $\chi(Y,\sO_Y) = 0$, which is a contradiction to (b).
\end{proof}

We now come to the main results of this section. There are two flavours: the first one deals with line bundles of numerical dimension $1$ on Kähler varieties, whereas the second one deals with parabolic line bundles on hyperkähler manifolds.

\begin{theorem}\label{thm:nu1}
Let $X$ be a $\Q$-factorial compact Kähler variety which is not uniruled. Let $\sM$ be a nef $\Q$-line bundle on $X$ such that $\nd(\sM)\leq 1$, and assume the $\Q$-line bundle $\sM\otimes\sO_X({-}K_X)$ is pseudoeffective. Let $\pi\colon Y\to X$ be a resolution of $X$, and let $T_{\min}\in\{\pi^*\sM\}$ be a current with minimal singularities.
\begin{enumerate}[\normalfont (a)]
\item If $\chi(Y,\sO_Y)\neq0$ and if all Lelong numbers of $T_{\min}$ are zero, then there exists a $\Q$-divisor $D\geq0$ such that $\{\sM\}=\{D\}$.
\item If not all Lelong numbers of $T_{\min}$ are zero, then there exists a $\Q$-divisor $D\geq0$ such that $\{\sM\}=\{D\}$, and the class $\{\pi^*\sM\}$ is rigid.
\end{enumerate}
\end{theorem}

\begin{proof}
We prove the theorem in several steps.

\medskip

\emph{Step 1.}
In this step we show (a). By the assumptions of (a) and by Skoda's lemma, we have $\mathcal I(mT_{\min})=\sO_Y$ for all positive integers $m$. Fix a Kähler class $\eta$ on $Y$. Then we conclude by Theorem \ref{thm:nu1algsing} applied to the good bilinear form $Q_\eta\colon H^{1,1}(Y,\R)\times H^{1,1}(Y,\R)\to\R$ from Proposition~\ref{prop:twoexamples}.

\medskip

\emph{Step 2.}
In the remainder of the proof we show (b). By the assumptions of (b) there exists a point $y\in Y$ such that $\nu:=\nu(T_{\min},y)>0$. Let $\mu\colon Z \to Y$ be the blowup of $Y$ at $y$, let $E$ be the exceptional divisor of $\mu$, and let $\eta'$ be a Kähler class on $Z$. Then by \cite[Corollary 4]{Fav99} we have
\begin{equation}\label{eq:Lelongineq}
\nu(\mu^*T_{\min},E)\geq\nu.
\end{equation}
Hence, if we denote
\begin{equation}\label{eq:59}
S:=\mu^*T_{\min}-\nu E,
\end{equation}
then we have $S\geq0$ by \eqref{eq:Lelongineq} and by considering the Siu decomposition of $\mu^*T_{\min}$, and
\begin{equation}\label{eq:898}
\{\mu^*\pi^*\sM\}=\{\nu E\}+\{S\}.
\end{equation}
Therefore, by Theorem~\ref{thm:nu1a}(a) applied to the good bilinear form 
$$Q_{\eta'}\colon H^{1,1}(Z,\R)\times H^{1,1}(Z,\R)\to\R,$$
there exists a $\Q$-divisor $D'\geq0$ on $Z$ such that $D'\in \{\mu^*\pi^*\sM\}$. Setting $D:=\pi_*\mu_*D'$, we deduce that $\{\sM\}=\{D\}$. This shows the first statement in (b).

\medskip

\emph{Step 3.}
It remains to show the second statement in (b). Assume first that $R\big(\{S\}\big)\neq0$.\footnote{Recall the Siu decomposition of the class $\{S\}$ as in Definition \ref{dfn:Siuclasses}.} Then Theorem~\ref{thm:nu1a}(b) applied to \eqref{eq:898} and to the good bilinear form $Q_{\eta'}$ gives that the class $\{\mu^*\pi^*\sM\}$ is Siu-residual. Since the current $\mu^*T_{\min}\in\{\mu^*\pi^*\sM\}$ has minimal singularities by \cite[Proposition~5.1]{LX24}, we conclude that the current $\mu^*T_{\min}$ has no divisorial part in its Siu decomposition, which contradicts \eqref{eq:Lelongineq}.

\medskip

\emph{Step 4.}
Therefore, $R\big(\{S\}\big)=0$. Let $S=R+G$ be the Siu decomposition of $S$, where $R$ is its residual part. Since $S$ is a current with minimal singularities in the class $\{S\}$ by \cite[Remark 5.4]{Laz24}, we have $R\in R\big(\{S\}\big)$, hence $R=0$ by Lemma \ref{lem:cohotozero}(b). Therefore, $\mu^*T_{\min}=G+\nu E$ by \eqref{eq:59}, and thus,
$$ D\big(\{\mu^*\pi^*\sM\}\big)=\{G+\nu E\}=\{\mu^*T_{\min}\}=\{\mu^*\pi^*\sM\}.$$
We deduce that the class $\{\mu^*\pi^*\sM\}$ is rigid by Remark \ref{rem:rigid2}, hence the class $\{\pi^*\sM\}$ is rigid by \cite[Corollary~4.2]{LX24}. This finishes the proof.
\end{proof}

\begin{remark}
In the setup of Theorem \ref{thm:nu1}, assume that there does not exist a $\Q$-divisor $D\geq0$ such that $\{\sM\}=\{D\}$. Then it follows from Steps 1 and~2 of the proof of Theorem~\ref{thm:nu1} (where the assumption that $T_{\min}$ is a current with minimal singularities was not used) that for any resolution $\pi\colon Y\to X$ we have $\chi(Y,\sO_Y)=0$, and that all Lelong numbers of any closed positive current $T\in \{\pi^*\sM\}$ vanish.
\end{remark}

Now we have:

\begin{proof}[Proof of Theorem \ref{thm:main1}]
Since klt singularities are rational by \cite[Theo\-rem~VII.1.1 and Lemma VII.1.7]{Nak04} and since $\chi(X,\sO_X)\neq 0$ by assumption, we have that $\chi(Y,\sO_Y)\neq 0$ for every resolution of singularities $Y\to X$. Thus, we may apply Theorem \ref{thm:nu1} for $\sM:=\sO_X(K_X+\Delta)$ and conclude that there exists a $\Q$-divisor $D\geq0$ such that $\{K_X+\Delta\}=\{D\}$.

Let $\pi\colon Y\to X$ be a log resolution of $(X,\Delta)$. Then there exist $\Q$-divisors $\Delta_Y \geq 0$ and $E\geq0$ without common components such that
\begin{equation}\label{eq:canonical log resolution}
K_Y+\Delta_Y\sim_\Q \pi^*(K_X+\Delta)+E,
\end{equation}
and the pair $(Y,\Delta_Y)$ is klt. Since $\{K_Y+\Delta_Y\}=\{\pi^*D+E\}$, we conclude that $\kappa(Y,K_Y+\Delta_Y)\geq0$ by \cite[Corollary 5.9]{Wan21}. Since by \eqref{eq:canonical log resolution} we have $\kappa(X,K_X+\Delta)=\kappa(Y,K_Y+\Delta_Y)$, the proof is complete.
\end{proof}

Corollary \ref{cor:main1} follows immediately from Theorems \ref{thm:main1} and \ref{thm:GM}. Theorem~\ref{thm:main2} follows immediately from Theorem \ref{thm:nu1} by rationality of klt singularities. Similarly as Corollary \ref{cor:main1}, one can deduce the following result in the case of Calabi--Yau $4$-folds.

\begin{corollary}\label{cor:main2}
Let $X$ be a $\Q$-factorial compact Kähler klt variety of dimension $4$ such that $X$ is not uniruled and $c_1(X)=0$. Let $\sL$ be a nef line bundle on $X$ of numerical dimension~$1$. Assume that $\chi(X,\sO_X)\neq0$. Let $\pi\colon Y\to X$ be a resolution and let $T_{\min}$ be a current with minimal singularities in the cohomology class of $\pi^*\sL$. If all Lelong numbers of $T_{\min}$ vanish, then there exists a semiample $\Q$-divisor $D$ on $X$ which belongs to the cohomology class of $\sL$.
\end{corollary}

\begin{remark}\label{remark bb decomposition}
One might be tempted to try to prove Theorem~\ref{thm:main2} by using the singular Bogomolov--Beauville decomposition theorem from \cite{HP19,BGL22}. That result says that there exists a quasi-\'etale cover $\xi\colon\widetilde{X}\to X$ such that $\widetilde X$ splits into a product of irreducible Calabi--Yau varieties (which are automatically projective), hyperkähler varieties and a torus. Then we have $\nd(\xi^*\sL)=1$; \emph{however}, the condition $\chi(X,\sO_X)\neq 0$ does not in general imply that $\chi(\widetilde X,\sO_{\widetilde X})\neq 0$. If we had $\chi(\widetilde X,\sO_{\widetilde X})\neq 0$, then an easy argument as in \cite[Section 8]{LP18} would allow to reduce to the case where $\widetilde X$ has only one factor. Then $\widetilde{X}$ would either be a K3 surface (by the assumption on the numerical dimension), or a projective Calabi--Yau variety, and Theorem~\ref{thm:main2} would follow either from classical results or from \cite{LP18}.
\end{remark}

Finally, in the following result we recover \cite[Theorem 4.1(ii)]{Ver10}, which we will need in the proof of Theorem \ref{thm:main3}.

\begin{theorem}\label{thm:parabolic}
Let $X$ be a hyperkähler manifold. Let $\sL$ be a nef $\Q$-line bundle on $X$ such that $q_X(\sL,\sL)=0$. Assume that there exists a closed positive current $T\in \{\sL\}$ such that all Lelong numbers of $T$ are zero. Then $\kappa(X,\sL)\geq0$.
\end{theorem}

\begin{proof}
By Skoda's lemma we have $\mathcal I(mT)=\sO_Y$ for all positive integers~$m$. Then we conclude by Theorem \ref{thm:nu1algsing} applied to the good bilinear form $q_X\vert_{H^{1,1}(X,\R)}\colon H^{1,1}(X,\R)\times H^{1,1}(X,\R)\to\R$, see Proposition~\ref{prop:twoexamples}.
\end{proof}

\section{Maximal Lelong components}\label{sec:Demailly}

In this section we prove Theorem \ref{thm:main3}. The proof is similar to that of \cite[Theorem 4.9]{LM23}, but we are able to bypass several approximation arguments by using a regularisation result from \cite{Dem92} directly, see Theorem \ref{thm:Dem92}.

We start by recalling the definition of Lelong components from \cite{Ver09}.

\begin{definition}\label{dfn:maxlelong}
Let $X$ be a compact Kähler manifold and let $T$ be a closed positive $(1,1)$-current on $X$. An irreducible analytic subset $Z\subseteq X$ is a \emph{Lelong component of $T$} if there exists $c_0>0$ such that $Z$ is an irreducible component of $E_c(T)$ for each $0 < c \leq c_0$. We say that $Z$ is a \emph{maximal Lelong component of $T$} if it is a Lelong component of $T$ of maximal dimension among all Lelong components.
\end{definition}

The following is the main technical result of this section.

\begin{theorem}  \label{thm:LelongHK}
Let $X$ be a hyperkähler manifold of dimension $2n$ and let $\alpha\in H^{1,1}(X,\R)$ be a nef class on $X$ with $q_X(\alpha, \alpha) = 0$. Let $T\in \alpha$ be a closed positive current and assume that not all Lelong numbers of $T$ are zero. Let $Z\subseteq X$ be a maximal Lelong component of $T$ and set $q:=\dim Z$. Then the following hold:
\begin{enumerate}[\normalfont (a)]
\item $\nd(\alpha)=n$ and $q\geq n$,
\item $(\alpha|_Z)^{q-n+1}=0$,
\item $Z$ is coisotropic with respect to the symplectic form,
\item if $q=n$, then $\alpha|_Z=0$ and $Z$ is Lagrangian,
\item if $T$ has minimal singularities, then there exists a positive integer $m_0$ such that for every $m\geq m_0$ and for every closed positive current $T_m \in \{m\alpha\}$, the complex subspace $V_m$ of $X$ defined by the multiplier ideal sheaf $\sI(T_m)$ has dimension at least $n$.
\end{enumerate}
\end{theorem}

The main technical input in the proof of Theorem \ref{thm:LelongHK} follows from the circle of ideas surrounding Demailly's regularisation techniques from \cite{Dem92}. In order to state the result we use properly, we need the following definition from \cite[Section 7]{Dem92}.

\begin{definition}\label{dfn:Lelongmax}
Let $X$ be a compact Kähler manifold of dimension $n$ and let $T$ be a closed positive $(1,1)$-current on $X$. To the current $T$ we associate the sequence 
$$ 0=b_1\leq\dots\leq b_n\leq b_{n+1}$$ 
of real numbers such that for each $p\in\{1,\dots,n\}$ and each $c\in(b_p,b_{p+1}]$ we have $\dim E_c(T)=n-p$, and $E_c(T)=\emptyset$ for $c>b_{n+1}$. Notice that $b_{n+1}<\infty$, i.e.\ $E_c(T)=\emptyset$ for $c\gg0$, since $X$ is compact. Let $\{Z_{p,k}\}_{k\in I_p}$ be the collection of all $(n-p)$-dimensional irreducible components of all sets $E_c(T)$ for $c\in(b_p,b_{p+1}]$, and set $\nu_{p,k} := \nu(T,Z_{p,k})$ for each $p$ and $k$.
\end{definition}

\begin{remark}
With notation from Definition \ref{dfn:Lelongmax} and by the definition of the Lelong set $E_c(T)$, we have $\nu_{p,k} > b_p$ for each $p$ and $k$. Note also that we can have $b_p=b_{p+1}$: for instance, if $E_c(T)$ has codimension at least two for every $c>0$, then $b_2=b_1=0$.
\end{remark}

The following is a more precise version of \cite[Theorem 1.7]{Dem92}, see Remark \ref{rem:stronglypositive} below.

\begin{theorem}\label{thm:Dem92}
Let $X$ be a compact K\"ahler manifold of dimension $n$, let $\eta$ be a Kähler class on $X$, and let $T$ be a closed positive $(1,1)$-current on $X$. Then, using the notation from Definition~\ref{dfn:Lelongmax}, there exists a positive integer $m$ such that the following holds.
\begin{enumerate}[\normalfont (a)]
\item For each $p=1,\dots,n$, there exists a closed strongly positive $(p,p)$-current $\Theta_p$ in the cohomology class $\big(\{T\}+b_1 m\eta\big)\cdot\ldots\cdot\big(\{T\}+b_p m\eta\big)$ such that
$$\Theta_p\geq \sum_{k\geq1}(\nu_{p,k}-b_1)\cdot\ldots\cdot (\nu_{p,k}-b_p)[Z_{p,k}].$$
\item For each $p=1,\dots,n$ we have
\begin{align*}
\sum_{k\geq1}(\nu_{p,k}-b_1)\cdot\ldots\cdot &\, (\nu_{p,k}-b_p) \{Z_{p,k}\}\cdot\eta^{n-p}\\
&\leq \big(\{T\}+b_1 m\eta\big)\cdot\ldots\cdot\big(\{T\}+b_p m\eta\big)\cdot\eta^{n-p}.
\end{align*}
\end{enumerate}
\end{theorem}

\begin{remark}\label{rem:stronglypositive}
In \cite[Theorem 1.7]{Dem92} it was shown that the current $\Theta_p$ is positive. However, we will need the stronger statement that it is even \emph{strongly positive}.
\end{remark}

\begin{proof}[Proof of Theorem \ref{thm:Dem92}]
We first show that (b) follows from (a). Indeed, fix a Kähler form $\omega\in\eta$. Since the form $\omega^{n-p}$ is strongly positive by \cite[Proposition III.1.11]{Dem12a}, and since the $(p,p)$-current
$$\Theta_p-\sum_{k\geq1}(\nu_{p,k}-b_1)\cdots(\nu_{p,k}-b_p)[Z_{p,k}]$$
is strongly positive by (a) and by Lemma \ref{lem:strongly1}, we have
$$\int_X\Theta_p\wedge\omega^{n-p}\geq \sum_{k\geq1}(\nu_{p,k}-b_1)\cdots(\nu_{p,k}-b_p)\int_X [Z_{p,k}]\wedge\omega^{n-p},$$
which is equivalent to (b).

We now prove (a). Consider the projective bundle $\pi\colon\mathbb P(T_X)\to X$ and set $\mathcal L:=\sO_{\mathbb P(T_X)}(1)$. Pick a positive integer $m$ such that the cohomology class $\{\mathcal L\}+m\pi^*\eta$ contains a K\"ahler form $\omega_\pi$. Then the form $\omega_\pi-\pi^*\omega$ defines a hermitian metric $h$ on $\mathcal L$ such that $\Theta_h(\mathcal L)+\pi^*\omega\geq0$.

We now follow the proof of \cite[Theorem 1.7]{Dem92} closely. The closed positive currents $\Theta_p$ defined in the proof of \cite[Theorem 1.7]{Dem92} satisfy the inequality in (a). It remains to show that  each $\Theta_p$ is strongly positive. We use the notation from that proof. If $p=1$, then the $(1,1)$-current $\Theta_1=T$ is strongly positive. If $p>1$, then the currents $\Theta_{p,c,k}$ in the proof of op.\ cit.\ are strongly positive by Lemma \ref{lem:strongly2}. Since the current $\Theta_p$ is defined as the weak limit of a some convergent subsequence in the family $\{\Theta_{p,c,k}\mid b_p<c\leq b_{p+1},\,k\geq1\}$, it is also strongly positive. This finishes the proof.
\end{proof}

As an immediate and possibly surprising corollary we have:

\begin{corollary}\label{cor:maxLelong}
Let $X$ be a compact Kähler manifold of dimension $n$ and let~$T$ be a closed positive $(1,1)$-current on $X$ such that the class $\{T\}$ is nef and $\nd\big(\{T\}\big)=d$. Assume that not all Lelong numbers of $T$ are zero.
\begin{enumerate}[\normalfont (a)]
\item Let $Z\subseteq X$ be a maximal Lelong component of $T$. Then $\dim Z\geq n-d$.
\item If $T$ has minimal singularities, then there exists a positive integer $m_0$ such that for every integer $m\geq m_0$ and for every closed positive current $T_m \in \{mT\}$, the complex subspace $V_m$ of $X$ defined by the multiplier ideal sheaf $\sI(T_m)$ has dimension at least $n-d$.
\end{enumerate}
\end{corollary}

\begin{proof}
We first show (a). Set $p:=n-\dim Z$ and let $\eta$ be a Kähler class on $X$. Then, with notation from Definition~\ref{dfn:Lelongmax}, we have $b_1=\ldots=b_p=0$ since $Z$ is a maximal Lelong component. Thus, Theorem \ref{thm:Dem92}(b) implies
$$\sum_{k\geq1}\nu_{p,k}^p \{Z_{p,k}\}\cdot \eta^{n-p}\leq \{T\}^p\cdot\eta^{n-p}.$$
Since $Z$ is one of the components $Z_{p,k}$, the left hand side of this inequality is strictly positive by Lemma \ref{lem:cohoclasses}(b), hence so is the right hand side. As $\{T\}^{d+1}=0$, we obtain $p\leq d$.

To show (b), pick a positive integer $m_0$ such that $m_0\nu(T,Z)\geq n$. Then for any $m\geq m_0$ and any closed positive current $T_m \in \{mT\}$ we have
$$\nu(T_m,Z)\geq\nu(mT,Z)=m\nu(T,Z)\geq m_0\nu(T,Z)\geq n,$$
where the first inequality holds since $mT$ has minimal sin\-gu\-la\-ri\-ties. Hence, we have $Z\subseteq V_m$ by Skoda's lemma, and we conclude by (a).
\end{proof}

Now we can prove Theorem \ref{thm:LelongHK}.

\begin{proof}[Proof of Theorem \ref{thm:LelongHK}]
Part (d) follows immediately from (b) and (c) by Remark~\ref{rem:coisotropic}. We have $\alpha\neq0$ by Lemma \ref{lem:cohotozero}(b), since $T\neq0$; thus, $\nd(\alpha)=n$ by \cite[Proposition 24.1]{GHJ03}. Now (a) and (e) follow immediately from Corollary~\ref{cor:maxLelong}.

We prove the remaining statements in the following two steps. We fix a Kähler class $\eta$ on $X$.

\medskip

\emph{Step 1.}
In this step we show (b). Set $p:=2n-q$. Then, with notation from Definition~\ref{dfn:Lelongmax}, we have $b_1=\ldots=b_p=0$ since $Z$ is a maximal Lelong component. Thus, Theorem \ref{thm:Dem92}(a) and Lemma \ref{lem:strongly1} imply that there exists a closed strongly positive $(p,p)$-current $\Theta_p\in\alpha^p$ and a closed strongly positive $(p,p)$-current $R$ such that
\begin{equation}\label{eq:dominate}
\Theta_p=\nu(T,Z)^p[Z]+R.
\end{equation}
Since the class $\alpha$ is nef and $\alpha^{n+1}=0$ by (a), by \eqref{eq:dominate} we have
\begin{align*}
0&=\alpha^{n+1}\cdot\eta^{n-1}=\{\Theta_p\}\cdot \alpha^{n-p+1}\cdot\eta^{n-1}\\
&=\nu(T,Z)^p\big\{Z\big\}\cdot \alpha^{n-p+1}\cdot \eta^{n-1}+\{R\}\cdot \alpha^{n-p+1}\cdot\eta^{n-1}\geq0.
\end{align*}
In particular,
\begin{equation}\label{eq:67898}
\big\{Z\big\}\cdot \alpha^{n-p+1}\cdot \eta^{n-1}=0.
\end{equation}

Denote $\beta:=\alpha|_Z$ and $\xi:=\eta|_Z$. Let $\pi\colon \widetilde Z \to Z$ be a resolution which is a composition of the normalisation $Z^\nu \to Z$ and a resolution $\widetilde Z \to Z^\nu$ obtained as a sequence of blowups along smooth centres contained in the singular locus. Then there exists a $\pi$-exceptional $\R$-divisor $E\geq0$ on $\widetilde{Z}$ such that ${-}E$ is relatively ample over $Z^\nu$, hence $\widetilde{\xi}:=\pi^*\xi-\{E\}$ is a Kähler class on $\widetilde{Z}$. 

Let $j\colon Z\to X$ be the inclusion morphism and denote $\tau:=\pi\circ j$. Then $\{Z\}=\tau_*\{\widetilde Z\}_\mathrm{fund}$, where $\{\widetilde Z\}_\mathrm{fund}$ is the fundamental class of the manifold $\widetilde{Z}$, see \cite[\S11.1.4]{Voi02}. By \eqref{eq:67898} and by the projection formula, we have
$$\big\{\widetilde Z\big\}_\mathrm{fund}\cdot \tau^*\alpha^{n-p+1}\cdot \tau^*\eta^{n-1}=0,$$
which is equivalent to
\begin{equation}\label{eq:6789}
\pi^*\beta^{n-p+1}\cdot \pi^*\xi^{n-1}=0.
\end{equation}

We claim that for each integer $0\leq k\leq n-1$ we have
$$ \pi^*\beta^{n-p+1}\cdot \pi^*\xi^{n-1-k}\cdot \widetilde{\xi}^k = 0. $$
Indeed, for $k=0$ the claim follows from \eqref{eq:6789}. Assume inductively that the claim holds for some $k$. Then
\begin{align*}
0 &= \pi^*\beta^{n-p+1}\cdot \pi^*\xi^{n-1-k}\cdot \widetilde{\xi}^k = \pi^*\beta^{n-p+1}\cdot \pi^*\xi^{n-2-k}\cdot\big(\widetilde{\xi}+\{E\}\big)\cdot \widetilde{\xi}^k\\
&= \pi^*\beta^{n-p+1}\cdot \pi^*\xi^{n-2-k}\cdot\widetilde{\xi}^{k+1} + \pi^*\beta^{n-p+1}\cdot \pi^*\xi^{n-2-k}\cdot \{E\}\cdot \widetilde{\xi}^k.
\end{align*}
Since $\pi^*\beta$, $\pi^*\xi$ and $\widetilde{\xi}$ are nef, both of the summands above have to vanish, which proves the claim. In particular, for $k=n-1$ we obtain
\begin{equation}\label{eq resolution integral zero}
\pi^*\beta^{n-p+1}\cdot \widetilde{\xi}^{n-1} = 0.
\end{equation}
This gives $\pi^*\beta^{n-p+1}=0$ by Lemma \ref{lem:cohoclasses}(b), hence $\beta^{n-p+1}=0$ by Lemma~\ref{lemma mixed hodge lemma}, as desired.

\medskip

\emph{Step 2.}
In this step we show (c). If $p=1$, then (c) holds by \cite[Proposition~I.6.5]{LM87}. Thus, we may assume that $p\geq2$. Let $\sigma$ be a holomorphic symplectic form on $X$ and let us denote the pullback of $\sigma$ to $Z$ by $\sigma_Z$. For a smooth point $z$ of $Z$, we will show that $\sigma_Z(z)^{n-p+1} =0$, which is equivalent to $Z$ being coisotropic with respect to $\sigma$ by Theorem \ref{theorem:coisotropic-characterisation}.

To that end, let $\omega\in\eta$ be a Kähler form. By Corollary \ref{cor:wedge}, we have
\begin{equation}\label{eq:n-p}
\{T\}^p \cdot \{\sigma\}^{n-p+1} \cdot \{\overline{\sigma}\}^{n-p+1}\cdot\eta^{p-2}=0\in H^{2n,2n}(X, \R).
\end{equation}
Let $\Theta_p$ be as in Step 1, and recall that it is strongly positive. Then by \cite[Corollary III.1.16]{Dem12a}, the current $\Theta_p \wedge \omega^{p-2}$ is strongly positive. The form $\sigma^{n-p+1} \wedge \overline{\sigma}^{n-p+1}$ is positive by \cite[Example~III.1.2]{Dem12a}, hence, if we denote
$$\rho:=\sigma^{n-p+1} \wedge \overline{\sigma}^{n-p+1}\wedge \omega^{p-2},$$
the current $\Theta_p \wedge \rho$ is positive by \cite[Corollary III.1.16]{Dem12a}. Since we have
$$\{\Theta_p \wedge \rho\}=\{T\}^p \cdot \{\sigma\}^{n-p+1} \cdot \{\overline{\sigma}\}^{n-p+1}\cdot\eta^{p-2},$$
we conclude by \eqref{eq:n-p} and Lemma \ref{lem:cohotozero}(b) that
\begin{equation}\label{eq:768}
\Theta_p \wedge \rho=0.
\end{equation}
Since the current $R$ from Step 1 is strongly positive, we have again by \cite[Corollary III.1.16]{Dem12a} that
$$ R \wedge \rho \geq 0. $$
Similarly, since the current $[Z]$ is strongly positive by \cite[III.\S1.20 and Theorem III.2.7]{Dem12a}, we have
$$ [Z] \wedge \rho \geq 0. $$
Combining these two inequalities with \eqref{eq:dominate} and \eqref{eq:768} yields
$$0 = \Theta_p \wedge \rho = R \wedge \rho + \nu(T,Z)^p  [Z] \wedge \rho \geq 0.$$
Therefore, $[Z] \wedge \rho=0$ as in the proof of Lemma \ref{lem:pseflines}. Hence, 
$$\int_{Z_\reg}\rho=0,$$
where $Z_\reg$ is the smooth locus of $Z$. Since $\rho|_{Z_\reg}$ is a positive form of top degree on $Z_\reg$ by \cite[Criterion~III.1.6]{Dem12a}, we deduce that $\rho|_{Z_\reg}=0$, thus $(\sigma^{n-p+1} \wedge \overline{\sigma}^{n-p+1})|_{Z_\reg}=0$ by Lemma \ref{lem:cohotozero}(a). We thus have for any $z\in Z^\reg$ that
$$\sigma_Z(z)^{n-p+1} \wedge \overline{\sigma}_Z(z)^{n-p+1}=0$$
as forms on $T_{Z,z}$. But then an easy calculation yields $\sigma_Z(z)^{n-p+1} =0$, as desired.
\end{proof}

\begin{proof}[Proof of Theorem \ref{thm:main3}]
Let $T\in\{\sL\}$ be a closed positive current. If all Lelong numbers of $T$ are zero, then $\kappa(X,\sL)\geq0$ by Theorem \ref{thm:parabolic}. Thus, we may assume that not all Lelong numbers of $T$ are zero, and let $Z\subseteq X$ be a maximal Lelong component of $T$. Then $\dim Z\geq n$ by Theorem \ref{thm:LelongHK}(a).

Assume first that $Z$ is a divisor on $X$. As $T-\nu(T,Z)\cdot Z\geq0$ in this case, there exists an effective $\Q$-divisor in $c_1(\sL)$ by Theorem \ref{thm:nu1a} applied to the good bilinear form $q_X|_{H^{1,1}(X,\R)}$. Since $h^1(X,\sO_X)=0$, this implies that $\kappa(X,\sL)\geq0$. Therefore, we may assume that $\dim Z\leq2n-2$. Then we conclude by Theorem~\ref{thm:LelongHK}.
\end{proof}

Corollary \ref{cor:main3} follows immediately from Theorems \ref{thm:parabolic} and \ref{thm:GM}. 

\begin{remark}
Analogues of Theorem~\ref{thm:LelongHK} and Theorem~\ref{thm:main3} are likely to hold in the more general context of primitive symplectic varieties in the sense of \cite[Definition~3.1]{BL22}. However, it is not clear how to adapt proofs (or maybe even the statements), especially concerning the use of good bilinear forms or regularisation techniques. Working on a resolution would require a much more detailed analysis and is postponed to a future work.
\end{remark}

\bibliographystyle{amsalpha}
\bibliography{biblio}

\end{document}